\documentclass[11pt]{article}

\usepackage[letterpaper,margin=1in]{geometry}

\usepackage[utf8]{inputenc}
\usepackage[USenglish]{babel}
\usepackage[T1]{fontenc}

\usepackage{lmodern}
\usepackage{amsmath}
\usepackage{bm}
\usepackage{amssymb}
\usepackage{mathtools}
\usepackage{amsthm}
\usepackage{latexsym}
\usepackage{fixmath}
\usepackage{upgreek}

\usepackage{cite}

\usepackage{titlesec}
\usepackage{hyperref}
\usepackage[capitalise,nameinlink]{cleveref}

\usepackage{graphicx}
\usepackage{float}
\usepackage{subcaption}
\numberwithin{equation}{section}

\newtheorem{theorem}{Theorem}[section]

\newtheorem{lemma}{Lemma}[section]
\newtheorem{remark}{Remark}[section]

\providecommand{\keywords}[1]
{
	\textbf{\textit{Keywords.}} {\small #1}
}

\allowdisplaybreaks

\usepackage[dvipsnames]{xcolor}

\usepackage{array}

\hypersetup{
  colorlinks=true,
  citecolor=blue,
	linkcolor=blue,
	filecolor=magenta,      
	urlcolor=cyan,
}

\def\div{\textnormal{div}}
\def\bV{\mathbf{V}}
\def\bv{\mathbf{v}}
\def\bu{\mathbf{u}}
\def\bn{\mathbf{n}}
\def\bq{\mathbf{q}}
\def\br{\mathbf{r}}
\def\bx{\mathbf{x}}
\def\bQ{\boldsymbol{Q}}
\def\bPi{\boldsymbol{\Pi}}
\def\bsi{\boldsymbol{\sigma}}
\def\btau{\boldsymbol{\tau}}
\def\bxi{\boldsymbol{\xi}}
\def\cK{\mathcal{K}}
\def\<{\langle}
\def\>{\rangle}

\def\esi{e_{\bsi}}
\def\etu{e_{\tilde\bu}}
\def\eu{e_{\bu}}
\def\eq{e_{\bq}}

\newcommand{\e}[1]{\cdot 10^{#1}}

\title{Mixed finite element projection methods for the unsteady Stokes equations}

\author{Costanza Aric\`o\thanks{Corresponding author; Department of Engineering, University of Palermo, viale delle Scienze, Palermo, 90128, Italy; email: {\tt costanza.arico@unipa.it}}
	\and
Rainer Helmig\thanks{Institute for Modelling Hydraulic and Environmental Systems (IWS), Department of Hydromechanics and Modelling of Hydrosystems, University of Stuttgart, Pfaffenwaldring 61, Stuttgart, D-70569, Germany; email: {\tt rainer.helmig@iws.uni-stuttgart.de}}
	\and
	Ivan Yotov\thanks{Department of Mathematics,
		University of Pittsburgh, 301 Thackeray Hall, Pittsburgh, PA 15260, USA;
		email: {\tt yotov@math.pitt.edu}.}}
\begin{document}

\maketitle

\begin{abstract}
We develop $H$(div)-conforming mixed finite element methods for the unsteady Stokes equations modeling single-phase incompressible fluid flow. A projection method in the framework of the incremental pressure correction methodology is applied, where a predictor and a corrector problems are sequentially solved, accounting for the viscous effects and incompressibility, respectively. The predictor problem is based on a stress-velocity mixed formulation, while the corrector projection problem uses a velocity-pressure mixed formulation. The scheme results in pointwise divergence-free velocity computed at the end of each time step. We establish unconditional stability and first order in time accuracy. In the implementation we focus on generally unstructured triangular grids. We employ a second order multipoint flux mixed finite element method based on the next-to-the-lowest order Raviart-Thomas space $RT_1$ and a suitable quadrature rule. In the predictor problem this approach allows for a local stress elimination, resulting in element-based systems for each velocity component with three degrees of freedom per element. Similarly, in the corrector problem, the velocity is locally eliminated and an element-based system for the pressure is solved. At the end of each time step we obtain a second order accurate $H$(div)-conforming piecewise linear velocity, which is pointwise divergence free. We present a series of numerical tests to illustrate the performance of the method.
\end{abstract}


\section{Introduction}

One of the main issues in the the numerical solution of the incompressible Stokes equations is that a saddle point problem needs to be solved, where the pressure plays the role of a Lagrange multiplier for imposing the divergence-free constraint. For the solution of unsteady incompressible flow one of the most successful approaches is provided by the ``projection methods'', which started in the late 1960s with the two independent works of Chorin \cite{Chorin} and Temam \cite{Temam}. Instead of solving a coupled saddle point system for the velocity and pressure unknowns at each time step, an intermediate velocity field is computed first, using the momentum equation and neglecting the incompressibility constraint. This intermediate velocity is projected to the space of (weakly) divergence-free velocity fields in the second step, to get the next update for the velocity and pressure. Detailed discussions of the projection methods can be found, for example, in the review papers \cite{BROWN2001464, GUERMOND2006}. 

The classical projection methods are based on the velocity-pressure formulation of the governing equations, which has long been the mainstream of incompressible flow computations. The associated finite element methods use inf-sup stable finite element pairs, such as the Taylor-Hood, the Crouzeix-Raviart, or the MINI elements \cite{BBF}. The projection step involves the solution of a Poisson problem for the pressure with continuous finite elements. As a result, the computed velocity is not pointwise divergence-free. In this paper we develop $H$(div)-conforming mixed finite element projection methods for the unsteady Stokes equations with pointwise divergence-free velocity. The predictor problem is based on a stress-velocity mixed formulation, while the corrector projection problem uses a velocity-pressure mixed formulation. We consider the Raviart-Thomas $RT_k$, $k \ge 0$ or Brezzi–Douglas–Marini $BDM_k$, $k \ge 1$ families of spaces for the viscous stress components in the predictor problem and the velocity in the projection problem. In either case, the computed velocity at the end of each step is $H$(div)-conforming and divergence-free piecewise polynomial of degree $\le k$. Our family of methods is motivated by the projection method developed in \cite{Arico-NS-RT0} for the Navier-Stokes equations, see also \cite{Arico-thermal,Arico-ODA}, which uses a finite volume method for the predicted velocity and $RT_0$-based two point flux approximation (TPFA) method for the corrected velocity. We note that there has been extensive work on developing numerical methods for Stokes flow with divergence-free velocity and/or pressure-robustness (independence of the velocity error of the viscosity and pressure), see e.g., \cite{Cockburn2004ALC, HU2024115449, Guzmn2013ConformingAD, Wang-Ye, LI2022114815,Volker-div-free}. These approaches are based on the velocity-pressure formulation and require solving a saddle point problem. Our projection method provides an efficient alternate to obtain a divergence-free velocity using a stable pair of mixed finite element spaces for the Poisson problem in the projection step. Furthermore, our proposed implementation through the use of multipoint flux mixed finite element methods requires solving only symmetric and positive definite systems.

The stress-velocity mixed finite element method for the Stokes and Navier-Stokes equations has received increased attention in recent years, see, e.g. \cite{Cai-MFE-NS,gos2011,gmor2014,Caucao2017AFF}. Advantages of this approach include unified framework for Newtonian and non-Newtonian flows, local momentum balance, and direct $H$(div)-conforming computation of the stress, the latter being paramount for modeling fluid flow around an obstacle. A disadvantage of this formulation is that it results in increased number of degrees of freedom. This issue is addressed in \cite{Caucao2020AMS}, where a multipoint stress mixed finite element method for Stokes flow is developed, based on the stress–velocity–vorticity formulation. There, a suitable choice of mixed finite element spaces and quadrature rules allows for a local elimination of the stress and vorticity, resulting in a positive definite system for the velocity. In all of the aforementioned works, the pressure is eliminated using the incompressibility constraint. As a result, mass conservation is only weakly imposed. In addition, the formulation involves the deviatoric part of the stress and the velocity approximation is discontinuous. In contrast, in our projection method, the pressure gradient in the momentum balance equation solved in the predictor problem is taken from the previous time step and does not need to be eliminated, hence the method involves the full stress. Moreover, while the predicted velocity is discontinuous, the corrected velocity computed in the projection step is $H$(div)-conforming and exactly divergence-free.

In the first part of the paper we develop the family of mixed finite element projection methods for the unsteady Stokes equations of general polynomial degree. We perform stability analysis for the fully discrete method, establishing unconditional stability. The analysis is complicated by the fact that the predicted and the corrected velocity live in different finite element spaces, the former being discontinuous, and the latter $H$(div)-conforming. Critical components of the analysis are the inclusion in the scheme of a variable approximating the pressure gradient, which lives in the space of the predicted velocity, and the fact that divergence-free vectors in either $BDM_K$ or $RT_k$ are piecewise polynomials of degree $\le k$. We then proceed with the time-discretization error analysis, focusing on the semi-discrete continuous-in-space scheme. We establish first order convergence in time for the viscous stress and the corrected velocity.

In the second part of the paper we develop a specific method from the family introduced in the first part, based on the next-to-the-lowest order Raviart-Thomas velocity space $RT_1$ and discontinuous piecewise linear polynomials on generally unstructured triangular grids. This results in a method that is second order accurate in space. In order to avoid solving saddle point problems, we employ the methodology of multipoint flux mixed finite element (MFMFE) methods developed in \cite{W-Y} for the Poisson problem. The approach is based on choosing a suitable combination of a finite element space for the vector flux variable and a quadrature rule for the associated bilinear form, which leads to a block-diagonal mass matrix with small local blocks. The flux can therefore be locally eliminated, resulting in a symmetric and positive definite algebraic system for the scalar potential variable. In \cite{W-Y}, the lowest order Brezzi–Douglas–Marini \(BDM_1\) space on simplicial or quadrilateral grids is utilized, together with a vertex quadrature rule. The blocks in the flux mass matrix are associated with mesh vertices and the resulting system for the potential variable, after local flux elimination, is of finite volume type with one constant value per element. The method is first order accurate in the $L^2$-norm and it is related to the multipoint flux approximation (MPFA) method \cite{Aavatsmark2002AnIT,Edwards1998FiniteVD}. A similar mass-lumping approach is studied in \cite{Brezzi.F;Fortin.M;Marini.L2006} and an alternative formulation based on a broken $RT_0$ velocity space is developed in \cite{Klausen-Winther-2006a}. The MFMFE methodology is extended to 3D hexahedral grids in \cite{iwy2010,wxy2012}, to higher order spaces in \cite{high-order-mfmfe,Radu}, and to the Stokes equations using a vorticity–velocity–pressure formulation in \cite{BOON2023108498}.
 
In our method we utilize the second order MFMFE methodology on triangular grids developed in \cite{Radu, tuprints21948}, which is based on the $RT_1$ flux space and a quadrature rule that samples the functions at the degrees of freedom, which are the normal components at each edge evaluated at the two edge endpoints and the vector evaluated at the element center of mass. We employ this approach in both the predictor and the projection problems. In the predictor problem, we split the momentum equation into two separate equations for the $x$- and $y$-components of the predicted velocity. In each sub-problem, we locally eliminate the viscous stress component and obtain a sparse symmetric and positive definite system for the predicted velocity component. Similarly, in the projection problem, we locally eliminate the corrected velocity and we solve a sparse symmetric and positive definite system for the pressure. All three systems involve three unknowns per element describing the linear variation within the element. The global stencil couples each element with all elements that share a vertex with it.
The viscous stress and corrected velocity are easily computed by local postprocessing. 
The resulting method is highly computationally efficient, involving the solution of three sparse symmetric and positive definite systems per time step, which is done with the preconditioned conjugate gradient method.  

We remark that in physical problems, including transport processes, fluid-structure interaction, free fluid-porous medium interface transfer, or turbulence modeling, the second order accuracy of the velocity and pressure would allow for a higher physical fidelity in the approximation. In this sense, the present paper could also be regarded as the seed of further extensions where more terms are included in the governing momentum equations, e.g., convective inertia or Brinkman term or coupling with transport or thermal effects, as in \cite{Arico-NS-RT0,Arico-thermal,Arico-ODA}.

The paper is organized as follows. In \cref{sect_ge} we present the governing equations and the numerical method. Stability and time discretization error analysis is performed in \cref{sec:analysis}. The second order multipoint flux mixed finite element method on triangular grids is developed in \cref{sec:mfmfe}. In \cref{sec:numer} we report the results of a series of numerical experiments that illustrate the performance of the method. These include verification of the second order convergence in space and first order convergence in time, two benchmark problems, where the results of the method are compared to numerical and analytical solutions provided in the literature, and a challenging application with a strongly irregular computational domain. We end with conclusions in \cref{sec:concl}.
 
\section{Governing equations and numerical method} \label{sect_ge}

We assume a Newtonian, single-phase, incompressible fluid with density \(\rho\). Let \(\Omega \subset \mathbf{R}^d \), $d=2,3$, be the computational domain. The unsteady Stokes equations are
\begin{subequations}  \label{eq:governing_Eqq}
\begin{align}
& \frac{\partial \mathbf{u}}{\partial t} - \nu \Delta \mathbf{u} + \nabla \Psi= 0 \quad \textnormal{in } \Omega \times (0,T], \label{eq:momentum} \\
& \nabla\cdot \mathbf{u} = 0 \ \ \mbox{in } \Omega \times (0,T], \label{eq:continuity}
\end{align}
\end{subequations}
where \(T\) is the final time, \(\mathbf{u}\) is the fluid velocity vector, \(\Psi = \frac{p}{\rho}\) is the kinematic pressure, \(\nu = \frac{\mu}{\rho}\) is the kinematic fluid viscosity and \(\mu\) the dynamic fluid viscosity. Let \(\Gamma\) be the boundary of \(\Omega\), with a unit outward normal vector \(\mathbf{n}\). Two types of boundary conditions are assigned over  \(\Gamma = \Gamma_d\cup\Gamma_n\), where \(\Gamma_d\) and \(\Gamma_n\) are the portions of \(\Gamma\) where we assign the velocity vector and the normal component of the stress, respectively. The problem is complemented with initial conditions for the velocity and pressure. The boundary and initial conditions are:
\begin{subequations} \label{eq:IBCs}
\begin{align}
& \mathbf{u}=\mathbf{u}_b \quad \textnormal{on } \ \Gamma_d, \ t \in [0,T], \\
  & (- \nu \nabla \mathbf{u} + \Psi \, \mathbf{I}) \, \mathbf{n} = \boldsymbol{\Sigma}_b \quad
  \textnormal{on } \ \Gamma_n, \ t \in [0,T], \label{eq:IBCs2} \\
& \mathbf{u}=\mathbf{u}_0 \quad \textrm{with} \quad \nabla \cdot \mathbf{u}_0=0, \quad \Psi = \Psi_0 \quad \textnormal{in } \ \Omega , \ t= 0,
\end{align}
\end{subequations}
where \(\mathbf{I}\) the identity matrix.

\subsection{Time discretization} \label{discretization}

We discretize the time interval $[0,T]$ by a partition with $N$ sub-intervals of length $\Delta t = T/ N$ with vertices $t_n$, $n = 0,\ldots,N$, $t_n = n \Delta t$. We denote by \(\varphi^n\) the value of variable \(\varphi\) computed at time $t_n$.
System \eqref{eq:governing_Eqq} is solved at each discrete time $t_{n+1}$ by applying an incremental pressure correction scheme \cite{GUERMOND2006}, where a predictor and a projection problem are solved sequentially, such that the time-discretization form of \eqref{eq:momentum} splits into
\begin{subequations} \label{FTS2}
\begin{align}
& \frac{\tilde{\mathbf{u}}^{n+1} - \mathbf{u}^n}{\Delta t} - \nu \Delta \tilde{\mathbf{u}}^{n+1} + \nabla \Psi^n  = 0,  \label{eq:Ps} \\
& \frac{\mathbf{u}^{n+1} - \tilde{\mathbf{u}}^{n+1}}{\Delta t} + \nabla \left(\Psi^{n+1} - \Psi^n \right) = 0, \label{eq:Cs}
\end{align}
\end{subequations}
where \eqref{eq:Ps} is the predictor problem and \eqref{eq:Cs} is the projection problem.

\subsection{Space discretization}\label{sec:space-discr}

The space discretization is based on a mixed variational formulation of the system \eqref{FTS2}. For this purpose we introduce the variables
$$
\boldsymbol{\sigma} = -\nu \, \nabla \mathbf{u}, \quad \mathbf{q} = \nabla \Psi.
$$
Here $\boldsymbol{\sigma}$ is the viscous pseudostress, a $d\times d$ tensor with rows $\sigma_i^T$, $\sigma_i = \nabla u_i$, $i = 1,\ldots, d$, where $u_i$ is the $i$-th component of the velocity $\bu$. With this notation the predictor problem \eqref{eq:Ps} can be written as
\begin{align}
  \frac{\tilde{\mathbf{u}}^{n+1} - \mathbf{u}^n}{\Delta t} + \nabla\cdot \bsi^{n+1}
  + \bq^{n} = 0.  \label{eq:Ps-1}   
\end{align}
In the variational formulation of \eqref{FTS2} we will use the space \(H(\textnormal{div},\Omega)\), with
\begin{equation}
H(\div,\Omega) = \left\{\mathbf{v}\in (L_2\left(\Omega\right))^d: \ \nabla \cdot \mathbf{v} \in L_2\left(\Omega\right) \right\},
\end{equation}
\noindent and define
\begin{equation}
\mathbf{V} = H(\div,\Omega), \quad \mathbf{V}_{0,\Gamma_s} = \left\{\mathbf{v} \in \mathbf{V} : \mathbf{v}\cdot \mathbf{n} = 0 \ \textnormal{on} \ \Gamma_s\right\},  \ s \in \{d,n\}, \quad W = L_2(\Omega).
\end{equation}  
For a domain \(D \subset \mathbf{R}^d\), the symbol \(\left(\cdot,\cdot\right)_D\) marks the \(L^2(D)\)-inner product and $\|\cdot\|_D$ denotes the $L^2(D)$-norm for scalar, vector, and tensor valued functions. We also use the standard Hilbert space notation $H^r(D)$ with a norm $\|\cdot\|_{r,D}$, where $r$ is a positive integer. We will omit the subscript \textit{D} if \(D = \Omega\). 
For $S \in \mathbf{R}^{d-1}$ we denote by $\< \cdot,\cdot \>_S$ the $L^2(S)$ inner product or duality pairing.

We discretize \(\Omega\) by a geometrically conforming partition $T_h$ (also called grid or mesh) made of $N_T$ non-overlapping triangles \(E\) called (computational) elements or cells. Two neighboring elements share a common side (or interface) \(e\). Let \(\mathfrak{S}_T\) be the number of edges.

Let \(\mathbf{V}_h \times W_h \subset \bV\times W\) be either the Raviart-Thomas (RT) or the Brezzi-Douglas-Marini (BDM) pairs of spaces \cite{BBF} on \(T_h\). These are pairs of inf-sup stable mixed finite element spaces with the property
\begin{equation}\label{div-Vh}
\nabla\cdot \bV_h = W_h.
\end{equation}
Let $(\bV_h)^d$ denote the tensor-valued space where each row is an element of $\bV_h$. Let $Q_h:L^2(\Omega) \to W_h$ and $\bQ_h:(L^2(\Omega))^d \to \bV_h$
be the $L^2$-orthogonal projection operator onto $W_h$, such that for any $w \in L^2(\Omega)$,
\begin{align}
(Q_h w - w, w_h) = 0 \ \ \forall w_h \in W_h. \label{Qh-defn}
\end{align}
Let $Q_h^\Gamma:L^2(\Gamma) \to \bV_h\cdot\bn$ and $\bQ_h^\Gamma:(L^2(\Gamma))^d \to (\bV_h)^d\,\bn$ be the $L^2$-orthogonal projection operators onto $\bV_h\cdot\bn$ and $(\bV_h)^d\,\bn$, respectively,
such that for any $\varphi \in L^2(\Gamma)$ and any $\bv \in (L^2(\Gamma))^d$,
\begin{align*}
&\<\varphi - Q_h^\Gamma\varphi,\bv_h\cdot\bn\>_{\Gamma} = 0 \ \ \forall \bv_h \in \bV_h, \\
&\<\bv - \bQ_h^\Gamma\bv,\btau_h \,\bn\> _{\Gamma} = 0 \ \ \forall \btau_h \in (\bV_h)^d.
\end{align*}
We will also use the mixed interpolant $\bPi_h:(H^1(\Omega))^d \to \bV_h$ \cite{BBF} satisfying for all $\bv \in (H^1(\Omega))^d$,
\begin{align}
&  (\nabla\cdot(\bPi_h \bv - \bv),w_h) = 0 \ \ \forall w_h \in W_h, \label{Pi-div}\\
& \<(\bPi_h \bv - \bv)\cdot\bn,\bv_h\cdot\bn\>_{\Gamma} = 0 \ \ \forall \bv_h \in \bV_h. \label{Pi-n}
\end{align}
We will also utilize the vector version of $Q_h$, $Q_h^d: (L^2(\Omega))^d \to (W_h)^d$ and the tensor version of $\bPi_h$, $\bPi_h^d:(H^1(\Omega))^{d\times d} \to (\bV_h)^d$.

\subsection{The mixed finite element projection method}\label{numerical_model}   

The numerical algorithm is as follows. 

\begin{itemize}

\item Initialization: Let $\bu_h^0 = \bPi_h \bu_0$, $\Psi_h^0 = Q_h \Psi_0$, $\bq_h^0 = Q_h^d \,\nabla \Psi_0$, $\Psi^0_{b} = \Psi^0|_{\Gamma}$.

\end{itemize}

For \(n = 0,\ldots, N-1\):

\begin{itemize}
  
\item Predictor problem (MFE discretization of \eqref{eq:Ps-1}):

Find \(\bsi_h^{n+1} \in (\mathbf{V}_h)^d  : \bsi_h^{n+1} \mathbf{n} = \bQ_h^\Gamma(\mathbf{\Sigma}_b^{n+1} - \Psi_b^n \mathbf{n})\) on \(\Gamma_n\) and \(\tilde{\mathbf{u}}_h^{n+1} \in (W_h)^d\), such that
\begin{subequations} \label{eq:Predproblem}
\begin{align}
  &\left( \nu^{-1} \bsi_h^{n+1}, \boldsymbol{\tau}_h \right)
  - \left( \tilde{\mathbf{u}}_h^{n+1}, \nabla \cdot \boldsymbol{\tau}_h \right) =
  - \langle \mathbf{u}_b^{n+1}, \boldsymbol{\tau}_h \, \mathbf{n} \rangle _{\Gamma_d} \quad
  \forall \boldsymbol{\tau}_h \in (\mathbf{V}_{h,0,\Gamma_n})^d, \label{eq:PP1} \\
  &\left( \frac{\tilde{\mathbf{u}}_h^{n+1} - \mathbf{u}_h^n}{\Delta t}, \boldsymbol{\xi_h} \right)
  + \left( \nabla \cdot \bsi_h^{n+1}, \boldsymbol{\xi_h}\right)
  + \left( \mathbf{q}_h^n, \boldsymbol{\xi_h}\right) = 0 \quad \forall \boldsymbol{\xi_h} \in (W_h)^d. 
\end{align}
\end{subequations}

\item Compute $\Psi_b^{n+1}$ on $\Gamma_n$:
$$
\Psi_b^{n+1} = \left(\mathbf{\Sigma}_b^{n+1} + \nu \nabla \tilde{\mathbf{u}}_h^{n+1} \mathbf{n} \right) \cdot \mathbf{n} \quad \textnormal{on} \quad \Gamma_n.
$$

\item Projection problem (MFE discretization of \eqref{eq:Cs}): 

Find \(\mathbf{u}_h^{n+1} \in \mathbf{V}_h : \mathbf{u}_h^{n+1} \cdot \mathbf{n} = Q_h^\Gamma\left(\mathbf{u}_b \cdot \mathbf{n}\right)\) on \(\Gamma_d\) and \(\Psi_h^{n+1} \in W_h\) such that
\begin{subequations} \label{eq:Projecproblem}
\begin{align}
  &  \hskip - .1cm \bigg( \frac{\mathbf{u}_h^{n+1}-\tilde{\mathbf{u}}_h^{n+1}}{\Delta t}, \mathbf{v}_h\bigg)
  - ( \Psi_h^{n+1} - \Psi_h^n,\nabla \cdot \mathbf{v}_h ) =
  - \langle \Psi_b^{n+1} - \Psi_b^n, \mathbf{v}_h \cdot \mathbf{n} \rangle_{\Gamma_n}
  \ \forall \mathbf{v}_h \in \mathbf{V}_{h,0,\Gamma_d}, \label{eq:Projecproblem_1} \\
  & \hskip -.1cm \left( \nabla \cdot \mathbf{u}_h^{n+1},w_h\right) = 0 \quad \forall w_h \in W_h.
  \label{eq:Projecproblem_2} 
\end{align}
\end{subequations}

\item Update the pressure gradient: Find \(\mathbf{q}_h^{n+1}\in (W_h)^d\) such that
\begin{gather} 
  \left( \mathbf{q}_h^{n+1}, \bxi_h \right) =  \left( \mathbf{q}_h^n, \bxi_h \right)
  - \left( \frac{\mathbf{u}_h^{n+1}-\tilde{\mathbf{u}}_h^{n+1} }{\Delta t} ,\bxi_h \right) \quad \forall \bxi_h \in (W_h)^d. \label{eq:up_q}
\end{gather}

\item Go to the next time step.
\end{itemize}

\begin{remark} \label{remark1}
  The predictor problem \eqref{eq:Predproblem} accounts for the viscous effects. It uses the pressure gradient $\bq^n$ computed at the previous time step. The intermediate velocity $\tilde{\mathbf{u}}_h^{n+1} \in (W_h)^d$ computed after the predictor problem is discontinuous. The problem also provides \(H(\div)\)-conforming viscous pseudostress $\bsi_h^{n+1} \in (\mathbf{V}_h)^d$.

The projection problem \eqref{eq:Projecproblem} computes the new pressure $\Psi_h^{n+1} \in W_h$ and corrects the velocity, resulting in \(H(\div)\)-conforming and pointwise divergence free $\bu_h^{n+1} \in \bV_h$. In particular,
due to \eqref{div-Vh}, \eqref{eq:Projecproblem_2} implies
\begin{equation}\label{div-free}
\nabla \cdot \mathbf{u}_h^{n+1} = 0.
\end{equation}
We note that for both the RT and BDM pairs, it holds that
\begin{equation}\label{uh-in-Wh}
\bu_h^{n+1} \in (W_h)^d.
\end{equation}
The above property is true for the BDM pairs, since $\bV_h$ and $W_h$ contain polynomials of the same degree. It also holds for the RT pairs, due to \eqref{div-free} and \cite[Corollary~2.3.1]{BBF}. 

The update on $\mathbf{q}_h^{n+1}$ \eqref{eq:up_q} provides an approximation of $\nabla \Psi(t_{n+1})$, since, together with \eqref{uh-in-Wh} and \eqref{eq:Projecproblem_1}, it implies
\begin{equation*} 
\mathbf{q}_h^{n+1} = \mathbf{q}_h^n - \frac{\mathbf{u}_h^{n+1}-\tilde{\mathbf{u}}_h^{n+1} }{\Delta t} \simeq \nabla \Psi(t_n) + \nabla \big(\Psi(t_{n+1})- \Psi(t_n)\big) = \nabla \Psi(t_{n+1}).
\end{equation*}

\end{remark}

\section{Stability and time discretization error analysis}\label{sec:analysis}

In this section we establish unconditional stability and first order accuracy in time for the projection method \eqref{eq:Predproblem}--\eqref{eq:up_q}. In the analysis we focus on the case of homogeneous Dirichlet boundary conditions $\bu = 0$ on $\Gamma$. We will use the notation
$$
\mathbf{V}_{0} = \left\{\mathbf{v} \in \mathbf{V} : \mathbf{v}\cdot \mathbf{n} = 0 \ \textnormal{on} \ \partial\Omega\right\}
$$
In this case, in order to have uniqueness of the pressure $\Psi$, its space is restricted to
$$
W_0 = \left\{w \in W: \int_\Omega w = 0\right\}.
$$
The corresponding mixed finite element spaces are denoted by $\bV_{h,0}$ and $W_{h,0}$. We note that for this choice of spaces, it holds that
\begin{equation}\label{div-Vh0}
\nabla\cdot \bV_{h,0} = W_{h,0}.
\end{equation}
With this choice of boundary conditions, the algorithm \eqref{eq:Predproblem}--\eqref{eq:up_q} takes the form

\begin{itemize}
  
\item Predictor problem:
Find $\bsi_h^{n+1} \in (\mathbf{V}_h)^d$
and \(\tilde{\mathbf{u}}_h^{n+1} \in (W_h)^d\), such that
\begin{subequations} \label{eq:Predproblem-0}
\begin{align}
  &\left( \nu^{-1} \bsi_h^{n+1}, \boldsymbol{\tau}_h \right)
  - \left( \tilde{\mathbf{u}}_h^{n+1}, \nabla \cdot \boldsymbol{\tau}_h \right) = 0
  \quad
  \forall \boldsymbol{\tau}_h \in (\mathbf{V}_{h})^d, \label{pred1} \\
  &\left( \frac{\tilde{\mathbf{u}}_h^{n+1} - \mathbf{u}_h^n}{\Delta t}, \boldsymbol{\xi_h} \right)
  + \left( \nabla \cdot \bsi_h^{n+1}, \boldsymbol{\xi_h}\right)
  + \left( \mathbf{q}_h^n, \boldsymbol{\xi_h}\right) = 0 \quad \forall \boldsymbol{\xi_h} \in (W_h)^d.
  \label{pred2}
\end{align}
\end{subequations}

\item Projection problem:
Find $\mathbf{u}_h^{n+1} \in \mathbf{V}_{h,0}$
  and \(\Psi_h^{n+1} \in W_{h,0}\) such that
\begin{subequations} \label{eq:Projecproblem-0}
\begin{align}
  &  \bigg( \frac{\mathbf{u}_h^{n+1}-\tilde{\mathbf{u}}_h^{n+1}}{\Delta t}, \mathbf{v}_h\bigg)
  - ( \Psi_h^{n+1} - \Psi_h^n,\nabla \cdot \mathbf{v}_h ) = 0
  \quad \forall \mathbf{v}_h \in \mathbf{V}_{h,0}, \label{eq:Projecproblem_1-0} \\
  & \left( \nabla \cdot \mathbf{u}_h^{n+1},w_h\right) = 0 \quad \forall w_h \in W_{h,0}.
  \label{eq:Projecproblem_2-0} 
\end{align}
\end{subequations}

\item Update the pressure gradient: Find \(\mathbf{q}_h^{n+1}\in (W_h)^d\) such that
\begin{gather} 
  \left( \mathbf{q}_h^{n+1}, \bxi_h \right) =  \left( \mathbf{q}_h^n, \bxi_h \right)
  - \left( \frac{\mathbf{u}_h^{n+1}-\tilde{\mathbf{u}}_h^{n+1} }{\Delta t} ,\bxi_h \right) \quad \forall \bxi_h \in (W_h)^d. \label{eq:up_q-0}
\end{gather}

\end{itemize}

\begin{remark}\label{rem:div-free}
  Due to \eqref{div-Vh0}, \eqref{eq:Projecproblem_2-0} implies that \eqref{div-free} still holds.
\end{remark}

\subsection{Stability analysis}
In the analysis, we will utilize the algebraic identity
\begin{equation}\label{a-bb}
a(a-b) = \frac12\left(a^2 - b^2 + (a-b)^2\right).
\end{equation}
Let $\bV^0 = \{\bv \in \bV: \nabla\cdot\bv = 0\}$ denote the divergence-free subspace of $\bV$, with a similar notation for $\bV_h^0$ and $\bV_{h,0}^0$. We have the following useful orthogonality property.

\begin{lemma}
For $\bq_h^{n+1}$ computed in \eqref{eq:up_q-0}, it holds that
\begin{equation}\label{q-orth}
(\bq_h^{n+1},\bv_h) = 0 \quad \forall \bv_h \in \bV_{h,0}^0.
\end{equation}
\end{lemma}

\begin{proof}
  Recalling the argument for \eqref{uh-in-Wh}, we have that $\bV_h^0 \subset (W_h)^d$, so
   we can test \eqref{eq:up_q-0} with $\bxi_h = \bv_h \in \bV_{h,0}^0$, which, together with \eqref{eq:Projecproblem_1-0}, implies
\begin{align*}
  (\bq_h^{n+1},\bv_h) & = (\bq_h^{n},\bv_h) = \ldots = (\bq_h^{0},\bv_h) = (Q_h^d\, \nabla \Psi_0,\bv_h)\\
  & = (\nabla \Psi_0,\bv_h) = - (\Psi_0,\nabla\cdot\bv_h) + \<\Psi_0,\bv_h\cdot\bn\>_{\Gamma} = 0.
\end{align*}
\end{proof}

We next establish a stability bound for the projection method.

\begin{theorem}\label{thm:stab}
For the method \eqref{eq:Predproblem-0}--\eqref{eq:up_q-0}, there exists a constant $C$ independent of $\Delta t $ and $h$ such that
\begin{equation}\label{eq:stab}
  \Delta t \sum_{n=0}^{N-1} \nu^{-1}\|\bsi_h^{n+1}\|^2 + \frac12\|\bu_h^N\|^2 + \frac12\sum_{n=0}^{N-1} \|\bu_h^{n+1} - \tilde\bu_h^{n+1}\|^2 + \frac{\Delta t^2}{2}\|\bq_h^N\|^2
  \le \frac12\|\bu_h^0\|^2 + \frac{\Delta t^2}{2}\|\bq_h^0\|^2.
\end{equation}
\end{theorem}

\begin{proof}
We start by taking test functions $(\btau_h,\bxi_h) = (\bsi_h^{n+1},\tilde\bu_h^{n+1})$ in \eqref{eq:Predproblem-0}, combining the equations, using \eqref{a-bb}, and multiplying by $\Delta t$, to obtain
\begin{equation}\label{eq:stab-1}
  \Delta t \, \nu^{-1}\|\bsi_h^{n+1}\|^2 + \frac12\|\tilde\bu_h^{n+1}\|^2 - \frac12\|\bu_h^{n}\|^2
  + \frac12\|\tilde\bu_h^{n+1} - \bu_h^{n}\|^2 + \Delta t \, (\bq_h^n,\tilde\bu_h^{n+1}) = 0.
\end{equation}
Next, we take $\bv_h = \bu_h^{n+1}$ in \eqref{eq:Projecproblem_1-0} and use \eqref{a-bb} and \eqref{div-free} to obtain
\begin{equation}\label{eq:stab-2}
\frac12\|\bu_h^{n+1}\|^2 - \frac12\|\tilde\bu_h^{n+1}\|^2
  + \frac12\|\bu_h^{n+1} - \tilde\bu_h^{n+1}\|^2 = 0.
\end{equation}
Taking $\bxi_h = \bq_h^{n+1}$, using \eqref{a-bb} and multiplying by $\Delta t^2$ results in
\begin{equation}\label{eq:stab-3}
  \frac{\Delta t^2}{2}\left(\|\bq_h^{n+1}\|^2 - \|\bq_h^{n}\|^2
  + \|\bq_h^{n+1} - \bq_h^{n}\|^2\right) = - \Delta t(\mathbf{u}_h^{n+1}-\tilde{\mathbf{u}}_h^{n+1},\bq_h^{n+1}) = \Delta t(\tilde{\mathbf{u}}_h^{n+1},\bq_h^{n+1}),
\end{equation}
where we used \eqref{q-orth} in the last equality. Summing \eqref{eq:stab-1}--\eqref{eq:stab-3}, we obtain
\begin{align}
  & \Delta t \, \nu^{-1}\|\bsi_h^{n+1}\|^2 + \frac12\left(\|\bu_h^{n+1}\|^2 - \|\bu_h^{n}\|^2\right) + \frac12\left(\|\tilde\bu_h^{n+1} - \bu_h^{n}\|^2 + \|\bu_h^{n+1} - \tilde\bu_h^{n+1}\|^2\right) \nonumber \\
  & \quad + \frac{\Delta t^2}{2}\left(\|\bq_h^{n+1}\|^2 - \|\bq_h^{n}\|^2
  + \|\bq_h^{n+1} - \bq_h^{n}\|^2\right) = \Delta t\, (\bq_h^{n+1} - \bq_h^n,\tilde\bu_h^{n+1}).
  \label{eq:stab-4}
\end{align}
For the term on the right hand side above, using \eqref{q-orth} and Young's inequality $ab \le \frac{\epsilon}{2}a^2 + \frac{1}{2\epsilon}b^2$ for any $\epsilon > 0$, we write
\begin{equation}\label{eq:stab-5}
  \Delta t\, (\bq_h^{n+1} - \bq_h^n,\tilde\bu_h^{n+1}) = \Delta t\, (\bq_h^{n+1} - \bq_h^n,\tilde\bu_h^{n+1} - \bu_h^n) \le \frac{\Delta t^2}{2}\|\bq_h^{n+1} - \bq_h^{n}\|^2
  + \frac12\|\tilde\bu_h^{n+1} - \bu_h^n\|^2.
\end{equation}
Bound \eqref{eq:stab} follows by combining \eqref{eq:stab-4}--\eqref{eq:stab-5} and summing over $n$ from $0$ to $N-1$.
\end{proof}

\subsection{Time discretization error analysis}

In this section we establish a bound on the time discretization error in the projection method. For this purpose, we consider the semi-discrete continuous-in-space formulation of the method
\eqref{eq:Predproblem-0}--\eqref{eq:up_q-0} given below.

\begin{itemize}

\item Initialization: Let $\bu^0 = \bu_0$, $\Psi^0 = \Psi_0$, $\bq^0 = \nabla \Psi_0$.
  
For \(n = 0,\ldots, N-1\):

\item  Predictor problem: Find $\bsi^{n+1} \in (\bV)^d$  and \(\tilde{\mathbf{u}}^{n+1} \in (W)^d\), such that
\begin{subequations} \label{eq:varPredproblem}
\begin{align}
  &\left( \nu^{-1} \bsi^{n+1}, \boldsymbol{\tau} \right)
  - \left( \tilde{\mathbf{u}}^{n+1}, \nabla \cdot \boldsymbol{\tau} \right) = 0
  \quad
  \forall \boldsymbol{\tau} \in (\mathbf{V})^d, \label{var1}\\
  &\left( \frac{\tilde{\mathbf{u}}^{n+1} - \mathbf{u}^n}{\Delta t}, \boldsymbol{\xi} \right)
  + \left( \nabla \cdot \bsi^{n+1}, \boldsymbol{\xi}\right)
  + \left( \mathbf{q}^n, \boldsymbol{\xi}\right) = 0 \quad \forall \boldsymbol{\xi} \in (W)^d. \label{var2}
\end{align}
\end{subequations}

\item Projection problem: Find $\mathbf{u}^{n+1} \in \mathbf{V}_0$
    and \(\Psi^{n+1} \in W_0\) such that
\begin{subequations} \label{eq:varProjecproblem}
\begin{align}
  &  \bigg( \frac{\mathbf{u}^{n+1}-\tilde{\mathbf{u}}^{n+1}}{\Delta t}, \mathbf{v}\bigg)
  - ( \Psi^{n+1} - \Psi^n,\nabla \cdot \mathbf{v} ) = 0
  \quad \forall \mathbf{v} \in \mathbf{V}_{0}, \label{eq:varProjecproblem_1} \\
  & \left( \nabla \cdot \mathbf{u}^{n+1},w\right) = 0 \quad \forall w \in W_0.
  \label{eq:varProjecproblem_2} 
\end{align}
\end{subequations}

\item Update the pressure gradient: Find \(\mathbf{q}^{n+1}\in (W)^d\) such that
\begin{gather} 
  \left( \mathbf{q}^{n+1}, \bxi \right) =  \left( \mathbf{q}^n, \bxi \right)
  - \left( \frac{\mathbf{u}^{n+1}-\tilde{\mathbf{u}}^{n+1} }{\Delta t} ,\bxi \right) \quad \forall \bxi \in (W)^d. \label{eq:var-up_q}
\end{gather}

\end{itemize}

As noted in Remark~\ref{rem:div-free}, \eqref{eq:varProjecproblem_2} implies that
\begin{equation}\label{div-free-0}
\nabla \cdot \mathbf{u}^{n+1} = 0.
\end{equation}

We next note that the solution to the model problem \eqref{eq:governing_Eqq} with boundary condition $\bu = 0$ on $\Gamma$ satisfies, for \(n = 0,\ldots, N-1\),
\begin{align}
  &\left( \nu^{-1} \boldsymbol{\sigma}(t_{n+1}), \boldsymbol{\tau} \right)
  - \left( \mathbf{u}(t_{n+1}), \nabla \cdot \boldsymbol{\tau} \right) = 0
  \quad \forall \boldsymbol{\tau} \in (\mathbf{V})^d, \label{true-1}\\
  & \left( \frac{\mathbf{u}(t_{n+1}) - \mathbf{u}(t_n)}{\Delta t}, \boldsymbol{\xi} \right)
  + \left( \nabla \cdot \boldsymbol{\sigma}(t_{n+1}), \boldsymbol{\xi}\right)
  + \left( \nabla \Psi(t_{n+1}), \boldsymbol{\xi}\right) = (T_{n+1}(\bu),\bxi) \quad \forall \boldsymbol{\xi} \in (W)^d, \label{true-2}\\
  & \nabla\cdot \bu(t_{n+1}) = 0, \label{true-3}
\end{align}
where
$$
T_{n+1}(\bu) = \frac{\mathbf{u}(t_{n+1}) - \mathbf{u}(t_n)}{\Delta t} - \frac{\partial \bu}{\partial t}(t_{n+1}).
$$

\begin{theorem}\label{thm:time-err}
For the semi-discrete method \eqref{eq:varPredproblem}--\eqref{eq:var-up_q}, assuming that the true solution is sufficiently smooth in time, there exists a constant $C$ independent of $\Delta t$ such that
\begin{align}\label{err-bound}
  & \Delta t \sum_{n=0}^{N-1}\nu^{-1}\|\bsi(t_{n+1}) - \bsi^{n+1}\|^2 + \|\bu(t_{N}) - \bu^N\|^2
  + \Delta t^2\|\nabla \Psi(t_N) - \bq^N\|^2 \le C \Delta t^2.
\end{align}

\end{theorem}

\begin{proof}
Let $\esi^{n} = \bsi(t_{n}) - \bsi^{n}$, $\etu^n = \bu(t_{n}) - \tilde\bu^n$, $\eu^n = \bu(t_{n}) - \bu^n$, and $\eq^n = \nabla \Psi(t_n) - \bq^n$.
Subtracting \eqref{var1}--\eqref{var2}, \eqref{div-free-0} from \eqref{true-1}--\eqref{true-3} gives the error equations
\begin{align}
  &\left( \nu^{-1} \esi^{n+1}, \boldsymbol{\tau} \right)
  - \left( \etu^{n+1}, \nabla \cdot \boldsymbol{\tau} \right) = 0
  \quad \forall \boldsymbol{\tau} \in (\mathbf{V})^d, \label{err-1}\\
  & \left( \frac{\etu^{n+1} - \eu^n}{\Delta t}, \boldsymbol{\xi} \right)
  + \left( \nabla \cdot \esi^{n+1}, \boldsymbol{\xi}\right)
  + \left( \eq^{n}, \boldsymbol{\xi}\right) = (T_{n+1}(\bu),\bxi) - (S_{n+1}(\nabla\Psi),\bxi) \quad \forall \boldsymbol{\xi} \in (W)^d, \label{err-2}\\
  & \nabla\cdot\eu^{n+1} = 0, \label{err-3}
\end{align}
where
$$
S_{n+1}(\nabla\Psi) = \nabla\Psi(t_{n+1}) - \nabla\Psi(t_{n}).
$$
Taking $(\btau,\bxi) = (\esi^{n+1},\etu^{n+1})$ in \eqref{err-1}--\eqref{err-2}, combining the equations, using \eqref{a-bb}, and multiplying by $\Delta t$, we obtain
\begin{align}
  & \Delta t \, \nu^{-1}\|\esi^{n+1}\|^2 + \frac12\|\etu^{n+1}\|^2 - \frac12\|\eu^{n}\|^2
  + \frac12\|\etu^{n+1} - \eu^{n}\|^2 + \Delta t \, (\eq^n,\etu^{n+1}) \nonumber \\
& \qquad = \Delta t(T_{n+1}(\bu),\etu^{n+1}) - \Delta t(S_{n+1}(\nabla\Psi),\etu^{n+1}). \label{eq:err-1}
\end{align}
Next, we subtract and add $\bu(t_{n+1})$ in \eqref{eq:varProjecproblem_1}, take $\bv = \eu^{n+1}$, and
use \eqref{err-3}, to obtain
$$
(\eu^{n+1} - \etu^{n+1},\eu^{n+1}) = \Delta t (\Psi^{n+1} - \Psi^n,\nabla\cdot\eu^{n+1}) = 0,
$$
which implies, using \eqref{a-bb}
\begin{equation}\label{eq:err-2}
\frac12\|\eu^{n+1}\|^2 - \frac12\|\etu^{n+1}\|^2
  + \frac12\|\eu^{n+1} - \etu^{n+1}\|^2 = 0.
\end{equation}
Next, subtracting and adding $\nabla \Psi(t_{n+1})$, $\nabla \Psi(t_{n})$, and $\bu(t_{n+1})$
in \eqref{eq:var-up_q}, multiplying by $\Delta t^2$, and taking $\bxi = \eq^{n+1}$ results in
\begin{equation}\label{eq:err-3a}
  \Delta t^2(\eq^{n+1} - \eq^n,\eq^{n+1}) + \Delta t (\eu^{n+1} - \etu^{n+1},\eq^{n+1})
  = \Delta t^2(S_{n+1}(\nabla\Psi),\eq^{n+1}).
\end{equation}
We observe that the argument for \eqref{q-orth} implies
$$
(\bq^{n},\bv) = 0 \quad \forall \bv \in \bV_{0}^0,
$$
which, combined with
\begin{equation}\label{grad-orth}
(\nabla \Psi,\bv) = -(\Psi,\nabla\cdot\bv) + \<\Psi,\bv\cdot\bn\>_{\Gamma} = 0 \quad \forall \bv \in \bV_{0}^0,
\end{equation}
implies
\begin{equation}\label{q-orth-err}
(\eq^{n},\bv) = 0 \quad \forall \bv \in \bV_{0}^0.
\end{equation}
Therefore, using that $\eu^{n+1} \in \bV_{0}^0$, cf. \eqref{err-3}, \eqref{eq:err-3a} reduces to
\begin{equation*}
  \Delta t^2(\eq^{n+1} - \eq^n,\eq^{n+1}) - \Delta t (\etu^{n+1},\eq^{n+1})
  = \Delta t^2(S_{n+1}(\nabla\Psi),\eq^{n+1}).
\end{equation*}
which, combined with \eqref{a-bb}, implies
\begin{equation}\label{eq:err-3}
  \frac{\Delta t^2}{2}\left(\|\eq^{n+1}\|^2 - \|\eq^{n}\|^2 + \|\eq^{n+1} - \eq^n\|^2\right)
 - \Delta t (\etu^{n+1},\eq^{n+1}) = \Delta t^2(S_{n+1}(\nabla\Psi),\eq^{n+1}).
\end{equation}
The next step is to sum \eqref{eq:err-1}, \eqref{eq:err-2}, and \eqref{eq:err-3}. For the sum of the last terms on left hand sides of \eqref{eq:err-1} and \eqref{eq:err-3}, using \eqref{q-orth-err}, we write
\begin{equation}\label{eq:err-4}
-(\etu^{n+1},\eq^{n+1} - \eq^n) = (\eu^{n} - \etu^{n+1},\eq^{n+1} - \eq^n)
\end{equation}
Thus, summing \eqref{eq:err-1}, \eqref{eq:err-2}, and \eqref{eq:err-3} and using \eqref{eq:err-4} we arrive at
\begin{align}
& \Delta t \, \nu^{-1}\|\esi^{n+1}\|^2 + \frac12\left(\|\eu^{n+1}\|^2 - \|\eu^{n}\|^2
+ \|\etu^{n+1} - \eu^{n}\|^2 + \|\eu^{n+1} - \etu^{n+1}\|^2 \right) \nonumber \\
& \qquad + \frac{\Delta t^2}{2}\left(\|\eq^{n+1}\|^2 - \|\eq^{n}\|^2 + \|\eq^{n+1} - \eq^n\|^2\right)
\nonumber \\
&\quad = - \Delta t (\eu^{n} - \etu^{n+1},\eq^{n+1} - \eq^n)
+ \Delta t(T_{n+1}(\bu),\etu^{n+1})  - \Delta t(S_{n+1}(\nabla\Psi),\etu^{n+1}) \nonumber \\
& \qquad + \Delta t^2(S_{n+1}(\nabla\Psi),\eq^{n+1}) =: I_1 + I_2 + I_3 + I_4. \label{eq:err-5}
\end{align}
We proceed with bounding the four terms on the right hand side. For $I_1$, using Young's inequality, we have
\begin{equation}\label{I1}
|I_1| \le \frac12\|\eu^{n} - \etu^{n+1}\|^2 + \frac{\Delta t^2}{2}\|\eq^{n+1} - \eq^n\|^2.
\end{equation}
For $I_2$, using Young's inequality, we write
\begin{align}
  |I_2| & = |\Delta t(T_{n+1}(\bu),\etu^{n+1} - \eu^{n+1}) + \Delta t(T_{n+1}(\bu),\eu^{n+1})| \nonumber \\
  & \le \Delta t^2\|T_{n+1}(\bu)\|^2 + \frac14\|\etu^{n+1} - \eu^{n+1}\|^2
  + \Delta t\|T_{n+1}(\bu)\|^2 + \frac{\Delta t}{4}\|\eu^{n+1}\|^2.\label{I2}
\end{align}
For $I_3$, using \eqref{grad-orth} and \eqref{err-3}, we have that $(S_{n+1}(\nabla\Psi),\eu^{n+1}) = 0$. Then, using Young's inequality, we obtain
\begin{equation}\label{I3}
|I_3| = |\Delta t(S_{n+1}(\nabla\Psi),\etu^{n+1} - \eu^{n+1})| \le \Delta t^2 \|S_{n+1}(\nabla\Psi)\|^2 + \frac14 \|\etu^{n+1} - \eu^{n+1}\|^2.
\end{equation}
For $I_4$, the use of Young's inequality gives
\begin{equation}\label{I4}
|I_4| \le \Delta t \|S_{n+1}(\nabla\Psi)\|^2 + \frac{\Delta t^3}{4}\|\eq^{n+1}\|^2.
\end{equation}
Combining \eqref{eq:err-5}--\eqref{I4} gives
\begin{align}
  & \Delta t \, \nu^{-1}\|\esi^{n+1}\|^2 + \frac12\left(\|\eu^{n+1}\|^2 - \|\eu^{n}\|^2\right)
  + \frac{\Delta t^2}{2}\left(\|\eq^{n+1}\|^2 - \|\eq^{n}\|^2\right)
\nonumber \\
&\quad \le \left(\Delta t^2 + \Delta t\right)\|T_{n+1}(\bu)\|^2
+ \left(\Delta t^2 + \Delta t \right)\|S_{n+1}(\nabla\Psi)\|^2
+ \frac{\Delta t}{4}\|\eu^{n+1}\|^2 + \frac{\Delta t^3}{4}\|\eq^{n+1}\|^2. \label{eq:err-6}
\end{align}
It is straightforward to show for the time discretization and splitting errors that $\|T_{n+1}(\bu)\| \le C\Delta t$ and $\|S_{n+1}(\nabla\Psi)\| \le C \Delta t$, assuming that the solution is smooth enough in time. Then \eqref{err-bound} follows by summing \eqref{eq:err-6} over $n$ from 0 to $N-1$, using that $\eu^0 = 0$ and $\eq^0 = 0$, and, assuming that $\Delta t \le 1$, applying the discrete Gronwall inequality \cite[Lemma~1.4.2]{QV-book} for the last two terms in \eqref{eq:err-6}.
\end{proof}

\section{A second order multipoint flux mixed finite element method}\label{sec:mfmfe}

We next describe our specific implementation of the mixed finite element projection method introduced in Section~\ref{sect_ge}. It is based on the next-to-the lowest order Raviart-Thomas spaces $RT_1$ on triangular grids \cite{R-T}, which results in a method with second order accuracy in space. Moreover, we employ the multipoint flux mixed finite element methodology developed in \cite{W-Y} as a first order method and extended in \cite{Radu} to a second order method using the $RT_1$ spaces \cite{R-T}. Applying a quadrature rule with nodes associated with the degrees of freedom of $\bV_h$, we obtain mass lumping in both the predictor and projection problems. In the predictor problem, we locally eliminate the viscous stress and solve a symmetric and positive definite system for each component of the predicted velocity. In the projection problem, we locally eliminate the corrected velocity and we solve a symmetric and positive definite system for the pressure. The viscous stress and corrected velocity are then easily computed by local postprocessing. The resulting method is computationally very efficient. It provides a second order accurate $H(\mbox{div})$-conforming velocity, which is linear on each triangle and pointwise divergence free.   

\subsection{Mixed finite element spaces} \label{Mapping}
We consider a partition $T_h$ of the computational domain into non-overlapping triangles, where $h$ is the maximal element diameter. Let \(\hat{E}\) be the reference triangle with vertices $\hat \br_1 = (0,0)^T$, $\hat \br_2 = (1,0)^T$, and $\hat \br_3 = (0,1)^T$. For any triangle \(E \in T_h\) with counterclockwise oriented vertices \(\br_i\), \(i = 1,2,3\), there exists a bijection linear mapping \(F_E : \hat{E} \rightarrow E\) such that $F_E(\hat\br_i) = \br_i$, $i=1,2,3$, given by
\begin{equation} \label{eq:mapping}
F_E(\hat\bx) = \mathbf{r}_1\left(1 - \hat{x} - \hat{y}\right) + \mathbf{r}_2 \hat{x} + \mathbf{r}_3 \hat{y}.
\end{equation}
We take the pair \(\left\{\mathbf{V}_h \times W_h\right\}\) on \(T_h\) introduced in \cref{sec:space-discr} to be the \(RT_1\) spaces \cite{R-T}. In the reference triangle \(\hat{E}\) these spaces are defined as
\begin{subequations}
\begin{gather}
\mathbf{V}(\hat{E}) = \big(P_1(\hat{E})\big)^2 + \mathbf{x}P_1(\hat{E}), \\
W(\hat{E}) = P_1(\hat{E}),
\end{gather}
\end{subequations}
with \(P_k(\hat{E})\) the space of bivariate polynomials of degree \(\le k\) on \(\hat{E}\). 
The dimension of $\mathbf{V}_1(\hat E)$ is 8. On any triangle, we can take two normal component degrees of freedom (DOFs) associated with each edge and two DOFs associated with the vector at the center of mass \cite{BBF}. On the reference triangle, we denote the center of mass as \( \hat{\mathbf{r}}_4 = \left(1/3, 1/3\right)^T \). We choose the two DOFs associated to each edge \(e\) be the values of \(\mathbf{v} \cdot \mathbf{n}\) at the two vertices of the edge \cite{W-Y}. This will allow us, through the use of a 
quadrature rule, to localize the interaction of the edge DOFs around mesh vertices. 
For $k = 1,2,3$, let \(\hat{\mathbf{n}}_{k,l}\), $l=1,2$, denote the outward unit normal vector to the two edges that share vertex $\hat\br_k$. For the center of mass $\hat\br_4$, let $\hat{\mathbf{n}}_{4,1} = (1 \,\ 0)^T$ and $\hat{\mathbf{n}}_{4,2} = (0 \,\ 1)^T$, see \cref{general}. For each of the points \(\hat{\mathbf{r}}_i, i = 1,\ldots,4\), the two associated \(RT_1\) basis functions \(\hat{\mathbf{v}}_{i,j}\), $j=1,2$, are defined by setting
\begin{equation} \label{eq1}
  \hat{\mathbf{v}}_{i,j}(\hat\br_k) \cdot \hat{\mathbf{n}}_{k,l} = \delta_{ik} \delta_{jl} , \qquad
i,k = 1,\ldots,4, \ \ j,l = 1,2.
\end{equation} 
where \(\delta_{\alpha\beta} = 1\) if \(\alpha = \beta\), \(\delta_{\alpha\beta} = 0\) if \(\alpha \neq \beta\). The basis functions are listed in \cref{settings}, see also \cite{Radu}, where in the second and third column of the table we give their \(x\) and \(y\) components.

\begin{figure}
	\centering
    \begin{subfigure}{0.33\linewidth}
		\includegraphics[width=\linewidth]{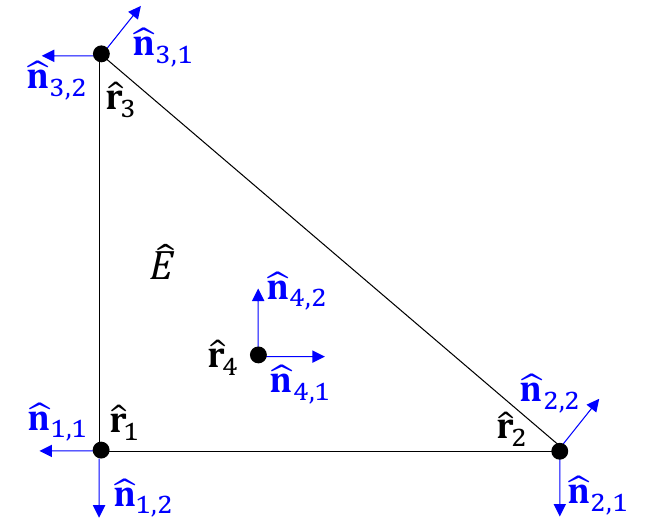}
		\caption{Degrees of freedom of $\bV(\hat E)$.}
		\label{general}
	\end{subfigure}
	\begin{subfigure}{0.66\linewidth}
	\centering
	\begin{tabular}{c|c c} 
		\hline 
		$\hat{\mathbf{v}}_{1,1}$ & $-2\hat{x}^2 -\hat{x}\hat{y} +3 \hat{x} + \hat{y} -1$ & $-\hat{y}^2 -2\hat{x}\hat{y} + \hat{y}$ \\ 
		$\hat{\mathbf{v}}_{1,2}$ & $-\hat{x}^2 -2\hat{x}\hat{y} + \hat{x} $ & $-2\hat{y}^2 -\hat{x}\hat{y} + \hat{x} + 3\hat{y} - 1$ \\ 
		$\hat{\mathbf{v}}_{2,1}$ & $\hat{x}^2 -\hat{x}\hat{y}$ & $-\hat{y}^2 +\hat{x}\hat{y} - \hat{x} + \hat{y}$  \\ 
		$\hat{\mathbf{v}}_{2,2}$ & $\sqrt{2}(2\hat{x}^2 +\hat{x}\hat{y} - \hat{x}$) & $\sqrt{2}(\hat{y}^2 +2\hat{x}\hat{y} - \hat{y})$ \\
		$\hat{\mathbf{v}}_{3,1}$ & $\sqrt{2}(\hat{x}^2 + 2\hat{x}\hat{y} - \hat{x})$ & $\sqrt{2}(2\hat{y}^2 +\hat{x}\hat{y} - \hat{y})$\\
		$\hat{\mathbf{v}}_{3,2}$ & $-\hat{x}^2 + \hat{x}\hat{y} + \hat{x} -\hat{y}$ & $\hat{y}^2 -\hat{x}\hat{y}$ \\ 
		$\hat{\mathbf{v}}_{4,1}$ & -3$\left(\hat{x} \hat{y}\right)$ - 6 $\left(\hat{x}^2 - \hat{x}\right)$  & -3$\left(\hat{y}^2 - \hat{y}\right)$ - 6 $\left(\hat{x} \hat{y}\right)$ \\ 
		$\hat{\mathbf{v}}_{4,2}$ & -6$\left(\hat{x} \hat{y}\right)$ - 3 $\left(\hat{x}^2 - \hat{x}\right)$  & -6$\left(\hat{y}^2 - \hat{y}\right)$ - 3 $\left(\hat{x} \hat{y}\right)$\\
		\hline
	\end{tabular}
        \bigskip
		\caption{Basis functions of $\bV(\hat E)$.}
		\label{settings}
	\end{subfigure}
	\caption{Degrees of freedom of $\bV(\hat E)$ and the associated basis functions.}
	\label{figure_Piola}
\end{figure}

The \(RT_1\) spaces on any element \(E \in T_h\) are defined via the transformations
\begin{subequations} \label{eq:transf}
\begin{gather}
  \mathbf{v} \longleftrightarrow  \mathbf{\hat{v}} : \mathbf{v} = \frac{1}{\left| \mathbf{J}_E\right|} \mathbf{J}_E \mathbf{\hat{v}}, \label{eq:Piola} \\
  w \longleftrightarrow \hat{w} : w = \hat{w},
\end{gather}
\end{subequations}
where \eqref{eq:Piola} is known as the Piola transformation. Here \(\mathbf{J}_E = \left[\mathbf{r}_{21}, \mathbf{r}_{31}\right]\) is the the Jacobian matrix associated with the mapping \eqref{eq:mapping}, with \(\mathbf{r}_{ij} = \mathbf{r}_i - \mathbf{r}_j\). The determinant of the Jacobian matrix is \(\left| \mathbf{J}_E\right| = 2 \left|E\right|\), and \(\left|E\right|\) is the area of triangle \(E\).
The Piola transformation has the properties \cite{BBF}
\begin{equation} \label{eq:Piola2}
  \nabla \cdot \mathbf{v} = \frac{1}{\left| \mathbf{J}_E\right|} \nabla_{\hat{\mathbf{x}}} \cdot \mathbf{\hat{v}},
  \quad \bv\cdot\bn_e = \frac{1}{|e|}\hat\bv\cdot\hat\bn_e,
\end{equation}
where $F_E: \hat e \to e$, $\bn_e$ and $\hat\bn_e$ are the unit normal vectors on the edges $e$ and $\hat e$, respectively, and $|e|$ is the length of $e$. Since the Piola transformation preserves the normal components of the vector, it is suitable to enforce the continuity of \(\mathbf{v} \cdot \mathbf{n}\) across any edge \(e\) \cite{BBF}.

The \(RT_1\) spaces are defined on \(T_h\) as
\begin{subequations}
\begin{gather}
\mathbf{V}_h= \left\{ \mathbf{v} \in \mathbf{V} : \mathbf{v}|_E \longleftrightarrow \hat{\mathbf{v}}, \ \hat{\mathbf{v}} \in \hat{\mathbf{V}} (\hat{E}) \ \forall E\in T_h \right\},  \\
W_h= \left\{ w \in W : w|_E \longleftrightarrow \hat{w}, \ \hat{w} \in \hat{W}(\hat{E}) \ \forall E\in T_h  \right\},
\end{gather}
\end{subequations}
where the transformations \(\mathbf{v} \longleftrightarrow \hat{\mathbf{v}}\) and \(w \longleftrightarrow \hat{w}\) are defined in \eqref{eq:transf}. In \cref{fig_Piola} we show the mapping $F_E: \hat E \to E$ with the DOFs of the $RT_1$ spaces, where the DOFs for \(W_h\), which is discontinuous piecewise linear, are the values at any three points within an element $E$.

\begin{figure}
\centering
\includegraphics[width=0.75\textwidth]{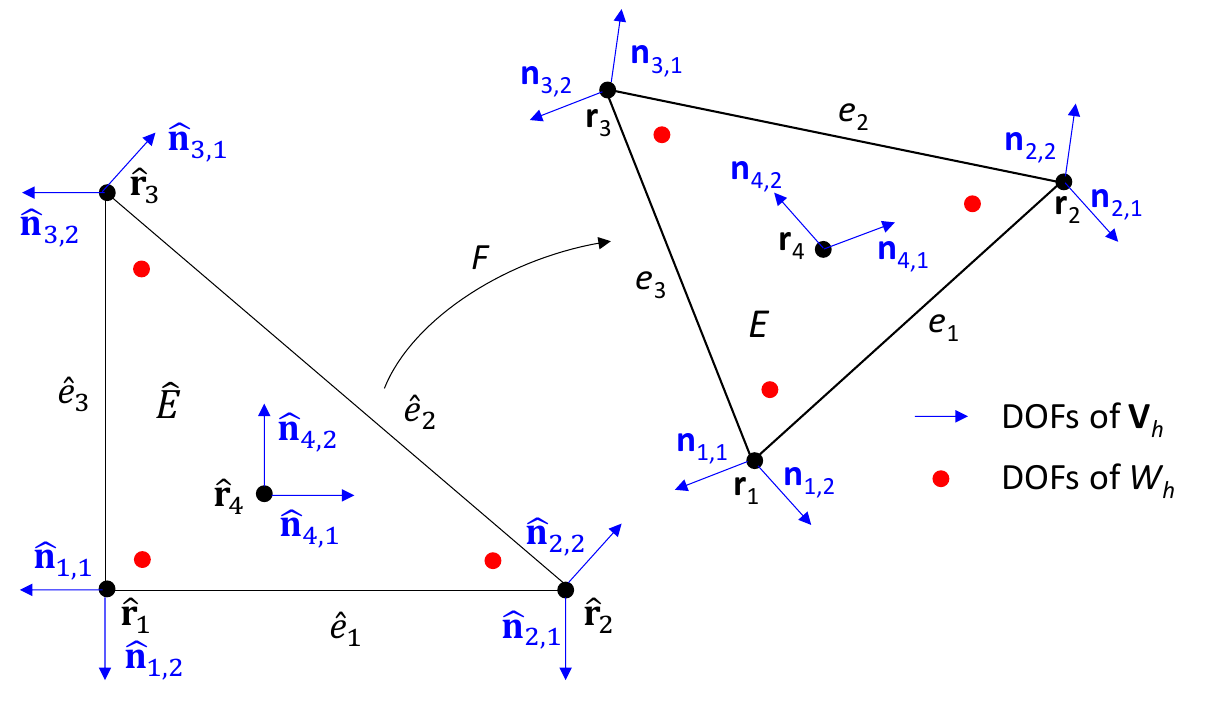}
\caption{Mapping from the reference triangle \(\hat{E}\) to a computational triangle \textit{E} with DOFs of \(RT_1\) in \(\hat{E}\) and \textit{E}.}
\label{fig_Piola}
\end{figure}

\subsection{The quadrature rule}\label{MPMFEM} 
Given any pair of vector functions \(\mathbf{s}, \mathbf{v} \in \bV_h\), we define the quadrature rule
\begin{equation} \label{eq:qr_elem}
  \left(\mathbf{s},\mathbf{v}\right)_Q = \sum_{E \in T_h} \left(\mathbf{s},\mathbf{v}\right)_{Q,E}
= \sum_{E \in T_h} |E| \sum_{i=1}^{4} \omega_i \, \mathbf{s}\left(\mathbf{r}_i\right) \cdot \mathbf{v}\left(\mathbf{r}_i\right),
\end{equation}
where the quadrature weights are $\omega_i = 1/12$, $i = 1,2,3$, and \(\omega_4 = 3/4\). This quadrature rule is proposed in \cite{Radu} and it is second order accurate. The vertex vector \(\mathbf{v}(\mathbf{r}_i)\), $i = 1,2,3$, can be uniquely obtained from the degrees of freedom $\bv(\br_i)\cdot\bn_{i,1}$ and $\bv(\br_i)\cdot\bn_{i,2}$, i.e., its normal components to the two edges sharing vertex \(\mathbf{r}_i\). Similarly, the vector at the center of mass $\br_4$ can be obtained from the degrees of freedom $\bv(\br_4)\cdot\bn_{4,1}$ and $\bv(\br_4)\cdot\bn_{4,2}$.

We apply the quadrature rule to the bilinear form \(( \nu^{-1} \bsi_h^{n+1}, \boldsymbol{\tau}_h)\) in \eqref{eq:PP1} in the prediction problem, and to the bilinear form \(\left(\frac{\mathbf{u}_h^{n+1}-\tilde{\mathbf{u}}_h^{n+1}}{\Delta t}, \mathbf{v}_h\right)\) in \eqref{eq:Projecproblem_1} in the projection problem. Since the quadrature rule \eqref{eq:qr_elem} couples only the two basis functions associated with the quadrature point $\br_i$, see \cite{W-Y} for more details, the viscous stress \(\boldsymbol{\sigma}^{n+1}_{x\left(y\right),h}\) and the velocity $\bu_h^{n+1}$ can be locally eliminated, resulting in symmetric and positive definite systems for $\tilde\bu_h^{k+1}$ in \eqref{eq:Predproblem} and $\Psi_h^{n+1}$ in \eqref{eq:Projecproblem}, respectively.

\subsection{Predictor problem}  \label{PP}
We solve separately for the \(x\) and \(y\) components of \eqref{eq:Predproblem}. Let $\bsi_{x,h}^{n+1}$ and $\bsi_{y,h}^{n+1}$ denote the first and second rows of $\bsi_{h}^{n+1}$, respectively. Let $\tilde u_{x,h}^{n+1}$ and $\tilde u_{y,h}^{n+1}$ denote the $x$ and $y$ components of $\tilde\bu_h^{n+1}$, with a similar notation for $q_{x,h}^n$ and $q_{y,h}^n$. For the boundary data, let $\boldsymbol{\Sigma}_b = (\Sigma_{x,b},\Sigma_{y,b})^T$ and $\bu_b = (u_{x,b},u_{y,b})^T)$. For brevity, we present our method only for the \(x\)-component. 

With the quadrature rule introduced in \eqref{eq:qr_elem}, the $x$-component of problem \eqref{eq:Predproblem} becomes: find \(\bsi_{x,h}^{n+1} \in \mathbf{V}_h  : \bsi_{x,h}^{n+1} \cdot \mathbf{n} = Q_h^\Gamma(\Sigma_{x,b}^{n+1} - \Psi_b^n \ n_x)\) on \(\Gamma_n\) and \(\tilde{u}_{x,h}^{n+1} \in W_h\), such that
\begin{subequations} \label{eq:Predproblem_x} 
\begin{align}
& \left( \nu^{-1} \bsi_{x,h}^{n+1}, \boldsymbol{\tau}_h \right)_Q - \left( \tilde{u}_{x,h}^{n+1}, \nabla \cdot \boldsymbol{\tau}_h \right) =- \langle u_{x,b}^{n+1}, \boldsymbol{\tau}_h \cdot \mathbf{n} \rangle _{\Gamma_d} \quad \forall \boldsymbol{\tau}_h \in \mathbf{V}_{h,0,\Gamma_n}, \label{eq:Predproblem_x_1} \\
& \left( \frac{\tilde{u}_{x,h}^{n+1} - u_{x,h}^n}{\Delta t}, \xi_h\right) + \left( \nabla \cdot \bsi_{x,h}^{n+1}, \xi_h\right) + \left( q_{x,h}^n, \xi_h\right) = 0 \quad \forall \xi_h \in W_h. \label{eq:Predproblem_x_2}
\end{align}
\end{subequations}

\subsubsection{Local stress elimination in the predictor problem} \label{reduction_PP}
In this subsection we describe the local elimination of the viscous stress $\bsi_{x,h}^{n+1}$ in the predictor problem \eqref{eq:Predproblem_x} in terms of the $x$-velocity $\tilde{u}_x^{n+1}$.

\paragraph{The case of any vertex.} \label{any_internal_vertex}

Consider first an interior vertex \(\mathbf{r}\) shared by \(\mathcal{K}\) edges and \(\mathcal{K}\) elements, see \cref{int_DOFs} (left), where $\mathcal{K} = 5$. Let $E_i$ and \(e_i\), $i = 1,\ldots,\mathcal{K}$, be, respectively, the elements and edges sharing \(\mathbf{r}\). Let \(\boldsymbol{\tau}_{i} \in \mathbf{V}_h\) be the basis functions on the edges $e_i$ associated with vertex $\br$, and let $\sigma_{x,i}^{n+1}$ be the associated DOFs of $\bsi^{n+1}_{x,h}$. 

\begin{figure}
  \begin{minipage}{.6\textwidth}
\centering
\includegraphics[width=\textwidth]{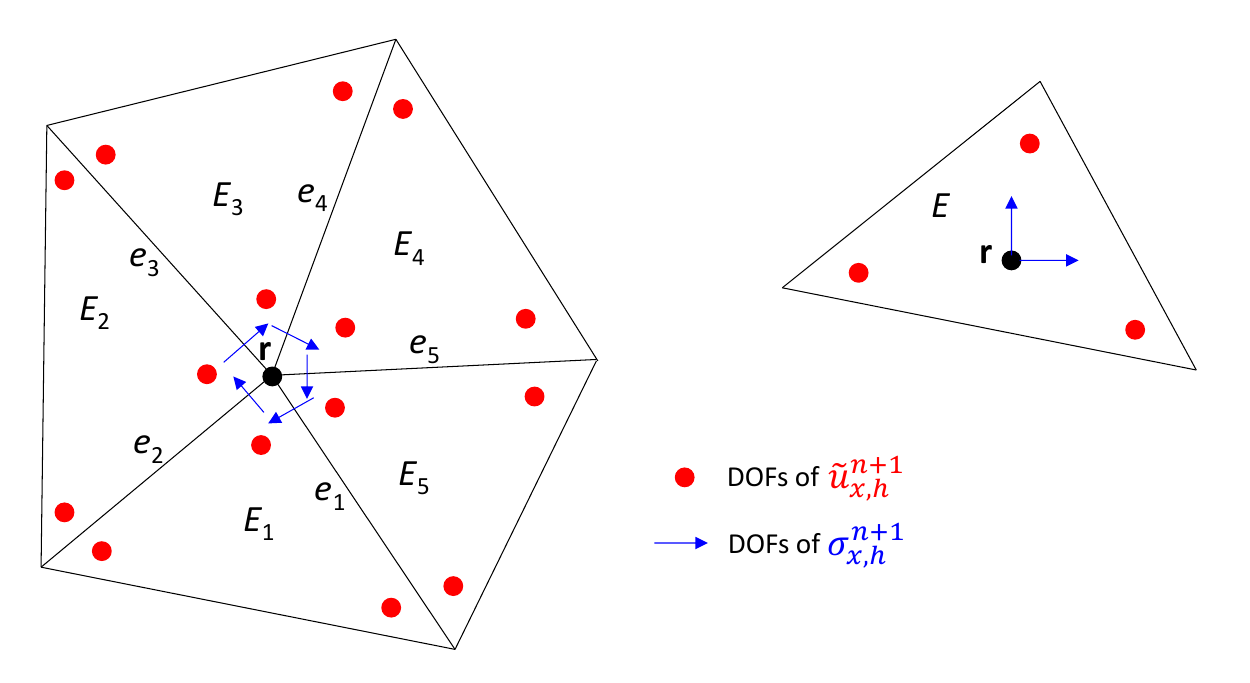}
  \end{minipage}
\hfill
  \begin{minipage}{.35\textwidth}
\centering  
\includegraphics[width=\textwidth]{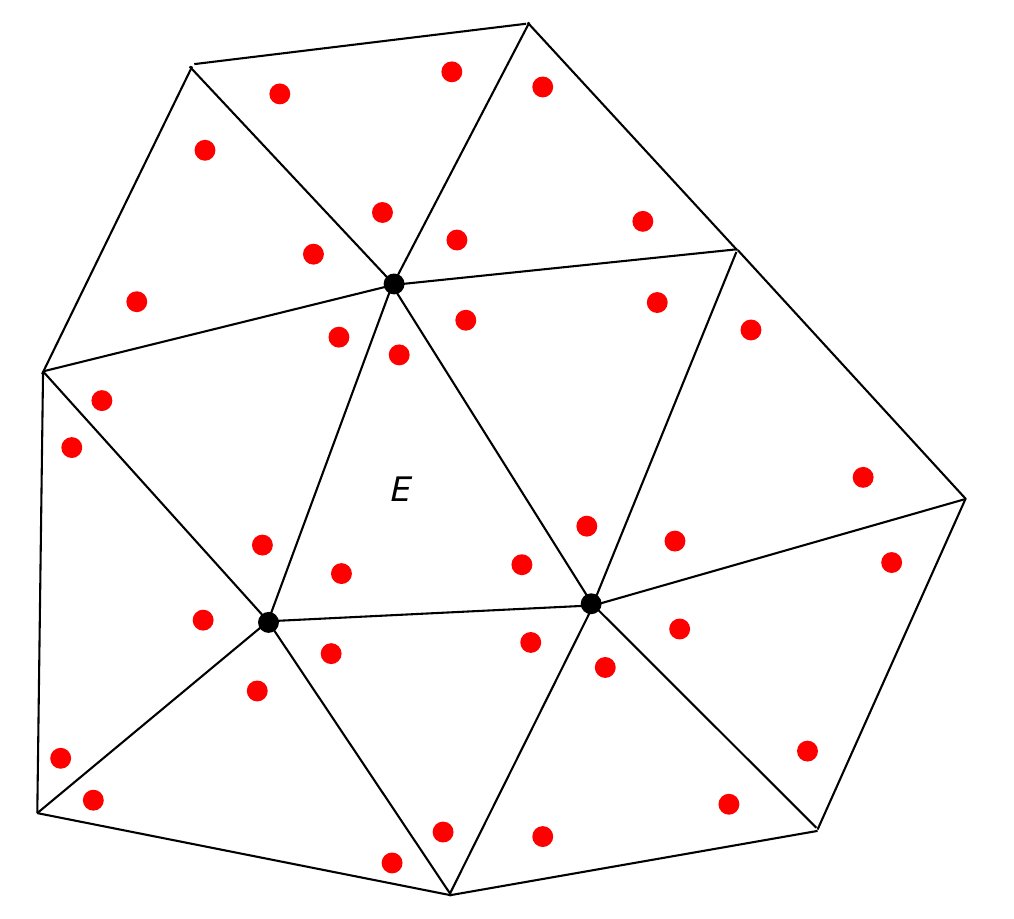}
\end{minipage}
\caption{Interaction of the DOFs of \(\bsi^{n+1}_{x,h}\) for a vertex (left) and a center of mass (center). Stencil of the reduced system for $\tilde\bu_h^{n+1}$ (right).}
\label{int_DOFs}
\end{figure}

Thanks to the properties of the quadrature rule discussed in \cref{MPMFEM}, the bilinear form $\left(\nu^{-1} \cdot, \cdot\right)_Q$ localizes the interactions of the DOFs of \(\bsi^{n+1}_{x,h}\). For example, taking \( \btau_h = \boldsymbol{\tau}_1 \) in \eqref{eq:Predproblem_x_1} couples \(\sigma^{n+1}_{x,1}\) only with \(\sigma^{n+1}_{x,2}\) and \(\sigma^{n+1}_{x,5}\). In a similar way, \(\sigma^{n+1}_{x,2}\) is coupled only with \(\sigma^{n+1}_{x,1}\) and \(\sigma^{n+1}_{x,3}\), and so on. 
Therefore, taking \(\boldsymbol{\tau}_h = \boldsymbol{\tau}_1, \ldots, \boldsymbol{\tau}_\mathcal{K}\) in \eqref{eq:Predproblem_x_1} results in a \( \mathcal{K} \times \mathcal{K} \) local linear system for the \(\sigma^{n+1}_{{x,i}}\) unknowns,  \(i=1, \dots, \mathcal{K}\):
\begin{equation} \label{eq:discr_x_0}
  \sum_{j=i-1}^{i+1} \sigma^{n+1}_{x,j} \left((\nu^{-1}\boldsymbol{\tau}_j,\boldsymbol{\tau}_i)_{Q,E_i}
+ (\nu^{-1}\boldsymbol{\tau}_j,\boldsymbol{\tau}_i)_{Q,E_{i+1}} \right)  
= (\tilde{u}^{n+1}_{x,h},\nabla \cdot \boldsymbol{\tau}_i)_{E_i}
+ (\tilde{u}^{n+1}_{x,h},\nabla \cdot \boldsymbol{\tau}_i)_{E_{i+1}}
, \quad i = 1, \dots, \mathcal{K},
\end{equation}
where we set $i - 1 := \mathcal{K}$ for $i = 1$ and $i + 1 := 1$ for $i = \mathcal{K}$.

We apply a 3-point Gaussian integration rule to compute \((\tilde{u}^{n+1}_{x,h},\nabla \cdot \boldsymbol{\tau}_h)_{E}\). It is exact for quadratic functions, and therefore it is exact for \((\tilde{u}^{n+1}_{x,h},\nabla \cdot \boldsymbol{\tau}_h)_{E}\), since \(\tilde{u}^{n+1}_{x,h} \in P_1\left(E\right)\) and \(\nabla \cdot \boldsymbol{\tau}_h \in P_1\left(E\right)\).

We write the local linear system \eqref{eq:discr_x_0} in a matrix-vector form as 
\begin{equation} \label{eq:vect_matr}
	\mathbf{A} \mathbf{\Sigma}_x = \mathbf{B}^T \mathbf{U}_x,
\end{equation} 
where 
\begin{itemize}
\item \(\mathbf{A}\) is a \(\left(\mathcal{K} \times \mathcal{K}\right)\) matrix with coefficients \(A_{i,j} = \left(\nu^{-1}\boldsymbol{\tau}_j,\boldsymbol{\tau}_i\right)_Q\), $i,j = 1,\ldots,\mathcal{K}$, 

\item \(\mathbf{\Sigma}_x\) is a \( \left(\mathcal{K} \times 1 \right) \) vector with coefficients \(\sigma^{n+1}_{x,i}\), \(i=1, \dots , \mathcal{K}\),

\item \(\mathbf{B}^T\) is a \(\left(\mathcal{K} \times 3\mathcal{K} \right)\) matrix, whose coefficients come from \((\tilde{u}^{n+1}_{x,h},\nabla \cdot \boldsymbol{\tau}_i)_{E_i}\) + \((\tilde{u}^{n+1}_{x,h},\nabla \cdot \boldsymbol{\tau}_i)_{E_{i+1}}\), $i = 1,\ldots,\mathcal{K}$,

\item \(\mathbf{U}_x\) is a \(\left(3\mathcal{K} \times 1\right)\) vector, whose coefficients are the three degrees of freedom of $\tilde{u}^{n+1}_{x,h}$ within each triangle $E_i$, $i = 1,\ldots,\mathcal{K}$.
\end{itemize}

The \( \left(\mathcal{K} \times \mathcal{K}\right) \) matrix \(\mathbf{A}\) is symmetric and positive definite, see \cite{W-Y}, and therefore the local system \eqref{eq:vect_matr} is solvable, which allows us to 
express the $\mathcal{K}$ normal components of the viscous stress DOFs \(\sigma^{n+1}_{x,i}\),  \(i=1, \dots, \mathcal{K}\) that share vertex $\br$ in terms of the $3\mathcal{K}$ DOFs of \(\tilde{u}^{n+1}_{x,h}\) on the $\mathcal{K}$ triangles that share vertex $\br$.

Next, consider a boundary vertex \(\mathbf{r}\). In this case the number of triangles that share it is $\cK-1$, so \(\mathbf{B}^T\) is a \(\left(\mathcal{K} \times 3(\mathcal{K}-1) \right)\) matrix and \(\mathbf{U}_x\) is a \(\left(3(\mathcal{K}-1) \times 1\right)\) vector. If $\br$ is on \(\overline\Gamma_d\), according to \eqref{eq:Predproblem_x_1}, in the local system \eqref{eq:vect_matr} we also have the contribution \(- \langle u^{n+1}_{x,b}, \boldsymbol{\tau}_h \cdot \mathbf{n} \rangle _{\Gamma_d}\). This is computed by numerical integration over the boundary edge(s) \(e_i\) on \(\Gamma_d\) that share $\br$ and system \eqref{eq:Predproblem_x} changes accordingly to
\begin{equation} \label{eq:vect_matr_gamma_d}
\mathbf{A} \mathbf{\Sigma}_x = \mathbf{B}^T \mathbf{U}_x + \mathbf{G}_{d,x} ,
\end{equation} 
where \(\mathbf{G}_{d,x}\) is a \(\left(\mathcal{K} \times 1 \right)\) vector with nonzero coefficients only for the rows corresponding to the boundary edge(s) on \(\Gamma_d\) sharing \(\mathbf{r}\). If \(\mathbf{r}\) is a boundary vertex on \(\overline\Gamma_n\), then the system \eqref{eq:Predproblem_x} is modified accordingly to incorporate the essential stress boundary condition, resulting in a system of type
\begin{equation} \label{eq:vect_matr_gamma_n}
\tilde{\mathbf{A}} \mathbf{\Sigma}_x = \mathbf{B}^T \mathbf{U}_x + \mathbf{G}_{n,x},
\end{equation} 
where \(\mathbf{G}_{n,x}\) is a \(\left(\mathcal{K} \times 1 \right)\) vector. 

\paragraph{The case of any center of mass.} \label{mp_vertex_PP}
We now consider any center of mass \(\mathbf{r}\) in triangle \textit{E}, see \cref{int_DOFs} (center). Let \(\boldsymbol{\tau}_i \in \mathbf{V}_h\), \(i = 1,2\), be the basis functions of \(\bsi^{n+1}_{x,h}\) associated with \(\mathbf{r}\). Let \(\sigma^{n+1}_{x,1}\) and \(\sigma^{n+1}_{x,2}\) be the two DOFs of $\bsi^{n+1}_{x,h}$ associated with $\br$, see \eqref{eq1}. Let \(\tilde{u}^{n+1}_{x,l}\), \(l=1,\dots,3\), be the three DOFs of $\tilde u_{x,h}^n$ in $E$. We recall that, due to the localization property of the quadrature rule $\left(\nu^{-1} \cdot, \cdot\right)_Q$, \(\sigma^{n+1}_{x,1}\) and \(\sigma^{n+1}_{x,2}\) are coupled only with each other. Therefore, taking \(\boldsymbol{\tau}_h = \boldsymbol{\tau}_1 \) and \(\boldsymbol{\tau} = \boldsymbol{\tau}_2 \) in \eqref{eq:Predproblem_x_1}, we obtain a \(2 \times 2\) linear system for \(\sigma^{n+1}_{x,1}\) and \(\sigma^{n+1}_{x,2}\):
\begin{equation} \label{eq:discr_x_01}
  \sum_{j=1,2} \sigma^{n+1}_{{x,j}}\left(\nu^{-1}\boldsymbol{\tau}_j,\btau_i\right)_{Q,E}
= \left(\tilde{u}^{n+1}_{x,h},\nabla \cdot \boldsymbol{\tau}_i\right)_{E}
  \qquad i = 1,2.
\end{equation}
We write the local linear system in a matrix-vector form as $\mathbf{A} \mathbf{\Sigma}_x = \mathbf{B}^T \mathbf{U}_x$, where
\begin{itemize}
\item \(\mathbf{A}\) is a $(2 \times 2)$ matrix with coefficients \(A_{i,j} = \left(\nu^{-1}\boldsymbol{\tau}_j,\boldsymbol{\tau}_i\right)_Q\), $i,j = 1,2$, 

\item \(\mathbf{\Sigma}_x\) is a \( \left(2 \times 1 \right) \) vector with coefficients \(\sigma^{n+1}_{x,i}\), \(i=1,2\),

\item \(\mathbf{B}^T\) is a \(\left(2 \times 3 \right)\) matrix, whose coefficients come from \((\tilde{u}^{n+1}_{x,h},\nabla \cdot \boldsymbol{\tau}_i)_{E}\), $i = 1,2$,

\item \(\mathbf{U}_x\) is a \(\left(3 \times 1\right)\) vector, whose coefficients are the velocity components \(\tilde{u}^{n+1}_{x,l}\), \(l=1,\ldots,3\).
\end{itemize}

As in the case of a vertex, the matrix matrix \(\mathbf{A}\) is symmetric and positive definite. Solving the $2\times 2$ system allows to express the two stress DOFs \(\sigma^{n+1}_{x,1}\) and \(\sigma^{n+1}_{x,2}\) that share $\br$ in terms of the three
velocity DOFs on $E$, \(\tilde{u}^{n+1}_{x,l}\), \(l=1,\dots,3\).

\subsubsection{Reduction of the predictor problem to a system for \(\tilde u_x\) and its solution}\label{assembly_PS}
With the local stress elimination presented in \cref{reduction_PP}, the MFMFE method in the predictor problem can be reduced to a system for \(\tilde u_{x,h}^{n+1}\) with three DOFs per element. We briefly describe this procedure.
The algebraic system arising from \eqref{eq:Predproblem_x} is
\begin{equation} \label{system_global}
		\begin{pmatrix}
			\mathbb{A} \quad -\mathbb{B}^T\\
		    \mathbb{B} \ \qquad \mathbb{D}
		\end{pmatrix} \begin{pmatrix}
		\boldmath{\Sigma}_x \\ \mathbb{U}_x
		\end{pmatrix} = \begin{pmatrix}
		\mathbb{G}_x \\ \mathbb{Q}_x + \mathbb{F}_x
		\end{pmatrix},
\end{equation}
where 

\begin{itemize}
\item \(\mathbb{A}\) is a \(\left(\mathfrak{K} \times \mathfrak{K}\right)\) block diagonal matrix, with \(\mathfrak{K} = 2 \times \mathfrak{S}_T + 2 \times N_T\), \(\mathfrak{S}_T\) is the total number of the edges and $N_T$ is the total number of triangles, as specified in \cref{sect_ge}, which is obtained by assembling the local block matrices \(\mathbf{A}\), described in \ref{any_internal_vertex} and \ref{mp_vertex_PP},

\item \(\boldmath{\Sigma}_x\) is a \(\left(\mathfrak{K} \times 1\right)\) vector with coefficients the DOFs of $\bsi_{x,h}^{n+1}$, given by assembling the local vectors \(\mathbf{\Sigma}_x\), defined in \ref{any_internal_vertex} and \ref{mp_vertex_PP},

\item \(\mathbb{B}^T\) is a \(\left(\mathfrak{K} \times 3N_T\right)\) matrix, given by assembling the block local matrices \(\mathbf{B}^T \), defined \ref{any_internal_vertex} and \ref{mp_vertex_PP},

\item \(\mathbb{D}\) is a \(\left(3N_T \times 3N_T\right)\) block diagonal matrix, with blocks associated with
  $\frac{1}{\Delta t}\left(\tilde{u}_{x,h}^{n+1},\xi_h\right)_E$,

\item \(\mathbb{U}_x\) is a \(\left(3N_T \times 1\right)\) vector with coefficients the DOFs of $\tilde{u}_{x,h}^{n+1}$,

\item \(\mathbb{G}_x\) is a \(\left(\mathfrak{K} \times 1\right)\) vector, given by assembling the local vectors \(\mathbf{G}_{d,x}\) and \(\mathbf{G}_{n,x}\), defined in \ref{any_internal_vertex},

\item \(\mathbb{Q}_x\) and \(\mathbb{F}_x\) are \(\left(3N_T \times 1\right)\) vectors corresponding to $-\left( q_{x,h}^n, \xi_h\right)$ and
  $\frac{1}{\Delta t}\left(u_{x,h}^{n},\xi_h\right)$, respectively. 
\end{itemize}

\begin{remark}
Similarly to the integrals in $\mathbb{B}$, we use a three-point Gaussian quadrature rule for computing the integrals in $\mathbb{D}$, \(\mathbb{Q}_x\), and \(\mathbb{F}_x\). All components in the associated bilinear forms are linear, recall \eqref{uh-in-Wh}, thus the integrands are quadratic. Therefore the computation of all integrals is exact.
\end{remark}

The stress unknown \(\boldmath{\Sigma}_x\) in \eqref{system_global} can be eliminated by inverting the matrix $\mathbb{A}$, which involves solving small local systems associated with the diagonal blocks of $\mathbb{A}$, as described in \ref{any_internal_vertex} and \ref{mp_vertex_PP}. This leads to a \(\left(3N_T \times 3N_T\right)\) system for \(\mathbb{U}_x\):
\begin{equation} \label{eq:sys_velox_PP}
	\left(\mathbb{D} + \mathbb{B} \ \mathbb{A}^{-1} \ \mathbb{B}^T \right)\mathbb{U}_x = \mathbb{Q}_x + \mathbb{F}_x + \mathbb{B} \mathbb{A}^{-1} \mathbb{G}_x,
\end{equation}
where the three DOFs of \(\tilde{u}^{n+1}_{x,h}\) within each triangle \(E\) are coupled with the three DOFs of all triangles that share a vertex with $E$, see \cref{int_DOFs} (right).

The matrix in \eqref{eq:sys_velox_PP} is symmetric and positive definite \cite{W-Y}. The system is solved applying a preconditioned conjugate gradient method with incomplete Cholesky factorization \cite{Dongarra}, resulting in a very efficient and fast procedure. The coefficients of the system matrix depend only on geometrical quantities, as well as on the fluid viscosity \(\nu\) and the time step size \(\Delta t\). Therefore the matrix is factorized only once, before the beginning of the time loop, saving a lot of computational effort.

We proceed in the same way along the \(y\) direction, solving a system with the same matrix in \eqref{eq:sys_velox_PP}.

At the end of the predictor problem we obtain the predicted velocity vector $\tilde{\bu}^{n+1}_{h}|_E \in (P_1(E))^2$. The normal flux continuity at the element interfaces is lost. 

\subsection{Computation of \(\Psi_b^{n+1} \) on \(\Gamma_n\)} \label{Psi_Gamma_n}
We compute the kinematic pressure \(\Psi_b^{n+1}\) on any edge $e \in \Gamma_n$ as
\begin{equation*}
\Psi_b^{n+1}|_e = \left(\boldsymbol{\Sigma}_b^{n+1}|_e + \nu \nabla \tilde{\mathbf{u}}_h^{n+1}|_e \, \mathbf{n}\right) \cdot \mathbf{n},
\end{equation*}
where the tensor $\nabla \tilde{\mathbf{u}}_h^{n+1}$ is computed on the element $E$ with edge $e$. We note that, since $\tilde{\mathbf{u}}_h^{n+1} \in (P_1(E))^2$, $\nabla \tilde{\mathbf{u}}_h^{n+1}$ is constant on $E$.

\subsection{Projection problem} \label{PjP}
Using the quadrature rule in \eqref{eq:qr_elem}, the discrete variational formulation of the  projection problem in \eqref{eq:Projecproblem} is: find \(\mathbf{u}_h^{n+1} \in \mathbf{V}_h : \mathbf{u}_h^{n+1} \cdot \mathbf{n} = Q_h^\Gamma\left(\mathbf{u}_b \cdot \mathbf{n}\right)\) on \(\Gamma_d\) and \(\Psi_h^{n+1} \in W_h\) such that
\begin{subequations} 
\begin{align*}
  &  \hskip - .1cm \bigg( \frac{\mathbf{u}_h^{n+1}-\tilde{\mathbf{u}}_h^{n+1}}{\Delta t}, \mathbf{v}_h\bigg)_Q
  - ( \Psi_h^{n+1} - \Psi_h^n,\nabla \cdot \mathbf{v}_h ) =
  - \langle \Psi_b^{n+1} - \Psi_b^n, \mathbf{v}_h \cdot \mathbf{n} \rangle_{\Gamma_n}
  \ \forall \mathbf{v}_h \in \mathbf{V}_{h,0,\Gamma_d}, \\ 
  & \hskip -.1cm \left( \nabla \cdot \mathbf{u}_h^{n+1},w_h\right) = 0 \quad \forall w_h \in W_h.
\end{align*}
\end{subequations}
The method is the same as the predictor problem method for $\bsi_{x,h}^{n+1} \in \mathbf{V}_h$ and $\tilde{u}_{x,h}^{n+1} \in W_h$ in \eqref{eq:Predproblem_x}. We apply the same local elimination procedure as described in \cref{reduction_PP,assembly_PS}, this time for the velocity $\mathbf{u}_h^{n+1}$, which results in a symmetric and positive system for $\Psi_h^{n+1}$ of type \eqref{eq:sys_velox_PP}, where the three DOFs of $\Psi_h^{n+1}$ within each triangle \(E\) are coupled with the three DOFs of all triangles that share a vertex with $E$. After solving the system, the velocity $\mathbf{u}_h^{n+1}$ is easily recovered by local postprocessing.

\subsection{Pressure gradient upgrade}

The last step is the pressure gradient update \eqref{eq:up_q}. Since $\mathbf{q}_h^{n+1}$ and the test function $\bxi_h$ are in the same discrete space $(W_h)^d$, this involves a simple update of the DOFs of $\mathbf{q}_h^{n+1}$. In particular, for any degree of freedom point $\br$,
$$
\mathbf{q}_h^{n+1}(\br) = \mathbf{q}_h^n(\br) - \frac{\mathbf{u}_h^{n+1}(\br) - \tilde{\mathbf{u}}_h^{n+1}(\br) }{\Delta t} .
$$

\begin{remark}
The method presented above requires solving at each time step three symmetric positive definite systems, two for the components of the predicted velocity in the predictor step, and one for the pressure in the projection step. Due to the local elimination of the viscous stress and the corrected velocity, the systems have only three unknowns per element, which is a significant reduction from the original saddle point systems. The resulting algorithm is computationally very efficient and obtains exactly divergence free velocity.
\end{remark}  

\section{Numerical tests}\label{sec:numer}

We present four tests to illustrate the behavior of the presented method. In the first test we study the convergence order in both space and time, under different boundary conditions (BCs) settings, considering both mildly and strongly distorted computational grids. The second test deals with the lid-driven cavity problem for several scenarios of the assigned BCs, where we compare our results with a solution provided in the recent literature. In the third test, we simulate Stokes flows in a confined channel with a variable geometry and compare the results with reference literature solutions obtained by applying the lubrication theory. In the last test, we investigate the capabilities of the proposed method in a ``real-world'' application with a strongly irregular computational domain.

The computational grids of the presented tests have been created by the open-source software Netgen \cite{Schberl1997NETGENAA}. An in-house code has been used to implement the method, and the open-source software Paraview \cite{Paraview} has been used to visualize the results.

\subsection{Test 1: convergence order in space and time} \label{test1}

We consider the square domain \(\left[-0.5, 0.5\right] \times \left[-0.5, 0.5\right]\) and the analytical solution
\begin{subequations} \label{eq:test1_analit}
\begin{align}
& u_x = -\cos(x) \sin(y) \sin(2  \pi  t)^2, \\
& u_y = -\sin(x) \cos(y) \sin(2  \pi  t)^2, \\
& \Psi = \frac{\cos(2 x) + \cos(2 y)}{4} \sin(2 \pi t)^2.
\end{align}
\end{subequations}

We consider the coarse mildly and badly distorted grids shown in \cref{fig_test1_mesh}, denoted by MDG and BDG, respectively. MDG has 85 vertices and 136 triangles, while BDG has 81 vertices and 128 triangles. For any triangle \(E\) of the grids, we consider the dimensionless quantity \(A_r\),
\begin{equation*}
A_r =\frac{h_{min}}{L_{max}} \frac{2}{\sqrt{3}},
\end{equation*}
where \(h_{min}\) and \(L_{max}\) are the minimum value of the heights and the maximum value of the lengths of \(E\), respectively, \(\frac{h_{min}}{L_{max}}\) is the aspect ratio of \(E\), and \(\frac{\sqrt{3}}{2}\) is the aspect ratio of the equilateral triangle. The number \(A_r\),
which is in the range \(\left[0, 1\right]\), represents the deviation of the aspect ratio of $E$ from the ideal value of an equilateral triangle. Dividing the interval \(\left[0, 1\right]\) into 19 sub-intervals with equal length, \cref{fig_test1_aspect_ratio} shows the number of triangles of MDG and BDG falling within each of the sub-intervals. As expected, for the MDG, the number of triangles is higher for the sub-ranges with higher values of \(A_r\), compared to the BDG. The minimum values of \(A_r\) are 0.36 and 0.21 for MDG and BDG, respectively.

\begin{figure}
	\centering
	\includegraphics[width=1\textwidth]{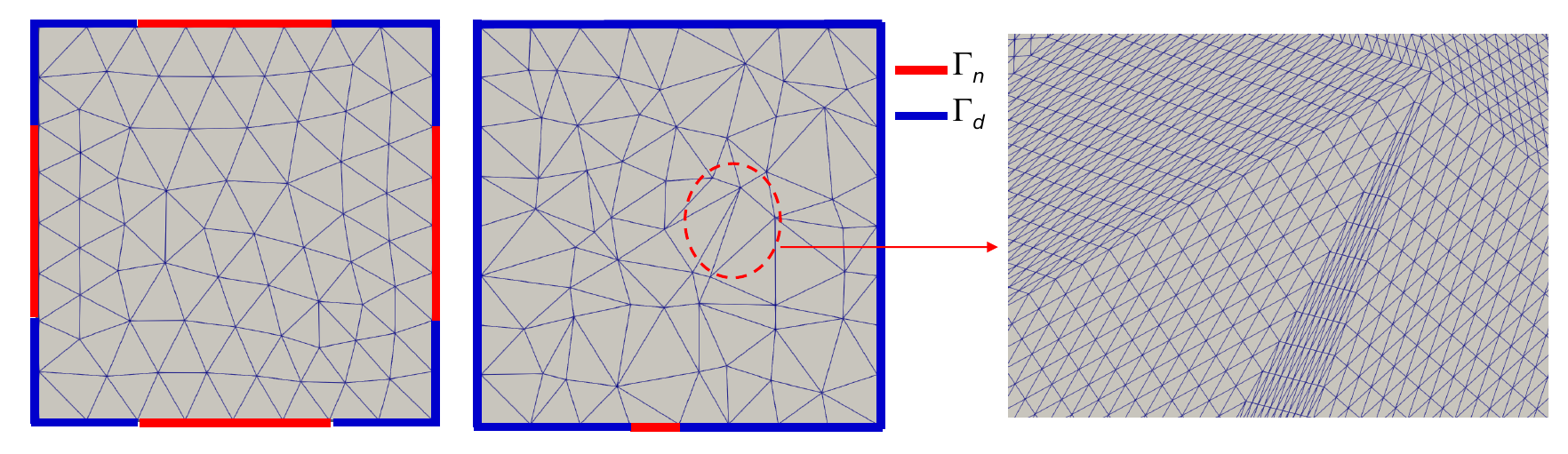}
	\caption{Test 1. Coarse computational grids and BCs: MDG and BCs scenario 1 (left), BDG and BCs scenario 2 (center). Right: Zoom of the BDG at the 5\(^{th}\) refinement level.}
	\label{fig_test1_mesh}
\end{figure}

\begin{figure}
	\centering
	\includegraphics[width=0.7\textwidth]{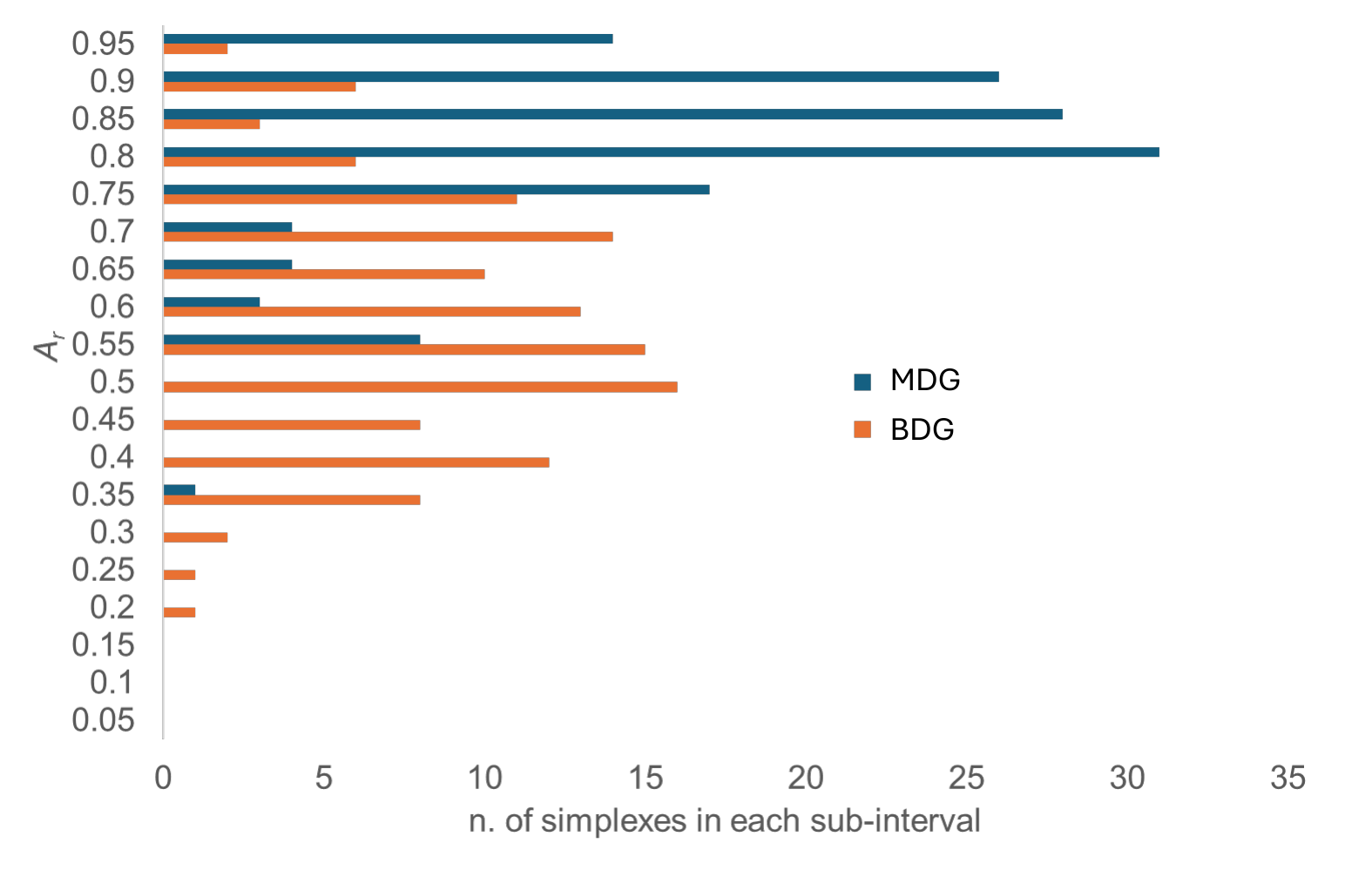}
	\caption{Test 1. Number of triangles of the coarse MDG and BDG falling in each sub-range of the interval \(\left[0, 1\right]\). }
	\label{fig_test1_aspect_ratio}
\end{figure}

Starting from the coarse MDG and BDG, we operate five refinements by halving each edge.
At each grid refinement, the number of triangles within each aspect ratio sub-interval in \cref{fig_test1_aspect_ratio} is multiplied by four. Below, the coarsest MDG and BDG are denoted as ``refinement level 0''.  \par

We consider two sets of BCs as depicted in \cref{fig_test1_mesh}. In the first case, there are alternating sections of \(\Gamma_d\) and \(\Gamma_n\), while in the second case only one boundary edge is on \(\Gamma_n\). Below, the two cases are denoted as ``BCs scenario 1'' and ``BCs scenario 2'', respectively. The second case is a way to implement full Dirichlet boundary condition, where the pressure is fixed at one edge to ensure uniqueness of the solution.\par

To analyze the convergence order in space, we choose a small time step size, \(\Delta t = 1\e{-4}\), and for each grid refinement level we compute the \(L_2\)-norms of the errors of the velocity components and kinematic pressure as
\begin{equation} \label{eq:L2_error}
  L_2^q = \max_{1\le n\le N}\|q(t_n) - q_h^n\|, \quad q=u_x, \, u_y, \, \Psi.
\end{equation}
Let \(h_l\) be the grid size of the \(l-\)th refinement level. If we assume that the error of variable \(q\) associated with the \(l-\)th grid refinement level is \(err_l^q \sim h_l^{r_h^q}\), the associated rate of convergence \(r_h^q\) is obtained by comparing the errors of the two consecutive refinement level grids with sizes \(h_l\) and \(h_{l+1}\):
\begin{equation*} 
r_h^q = \frac{\log \left(\frac{err_l^q}{err_{l+1}^q}\right)}{\log\left(\frac{h_l}{h_{l+1}}\right)}.    
\end{equation*}
In \cref{error_MDG_1} -- \cref{error_BDG_2} we list the \(L_2\) norms of the errors and the associated convergence order. The values of the errors obtained for the MDG are similar for the two scenarios of BCs. As expected, \(L_2^{u_x}\) and \(L_2^{u_y}\) are slightly smaller for BCs scenario 2, where all but one of the boundary edges are on \(\Gamma_d\), while \(L_2^{\Psi}\) is slightly higher in BCs scenario 2 than in BCs scenario 1. The convergence order is 2 both all three variables, \(L_2^{\Psi}\) and \(L_2^{u_x}\) and \(L_2^{u_y}\), see \cref{error_MDG_1} -- \cref{error_MDG_2}. The errors on the BDG, reported in \cref{error_BDG_1} -- \cref{error_BDG_2}, are generally 2 orders of magnitude higher than on the MDG. The order of convergence for the pressure is 2 and it is just below 2 for the velocity components on the finest grid levels. 

To study the convergence order in time, we use the finest grid (5-th refinement level) and a time step size $\Delta t$ in the range \(\left[1\e{-1}, 1\e{-4}\right]\). Again, we compute the \(L_2\)-norms of the errors as in \eqref{eq:L2_error}, and the convergence rates as
\begin{equation*}
r_{\Delta t}^q = \frac{\log \left(\frac{err_l^q}{err_{l+1}^q}\right)}{\log\left(\frac{\Delta t_l}{\Delta t_{l+1}}\right)},
\end{equation*}
where \(\Delta t_l\) and \(\Delta t_{l+1}\) are two consecutive values of the time step size in the selected range.  \cref{error_time_MDG_1} -- \cref{error_time_BDG_2} show the computed errors and convergence rates. The convergence rate is 1 for all variables, with slightly higher errors on the MDG. These results are consistent with the theoretical first order convergence for the time discretization error established in \cref{thm:time-err}. In addition, the results from all tables illustrate the unconditional in time and space stability of the method established in \cref{thm:stab}, as well as its robustness for highly distorted grids.

\begin{table} [H]
  \footnotesize
	\caption{Test 1. \(L_2\)-norms of errors and convergence orders in space: MDG, BCs scenario 1.}
	\centering
	\begin{tabular}{c c c c c c c}
		ref. level & \(L_2^{\Psi}\) & $L_2^{u_x}$ & $L_2^{u_y}$ & \(r_h^{\Psi}\) & $r_h^{u_x}$ & $r_h^{u_y}$  \\
		\hline
		0 & 1.142E-02 &	2.996E-04 &	2.972E-04 & & & \\
		1& 2.865E-03 &	7.606E-05 &	7.582E-05 &	1.995E+00 &	1.978E+00 &	1.971E+00 \\
		2 & 7.187E-04 &	1.923E-05 &	1.909E-05 &	1.995E+00 &	1.984E+00 &	1.990E+00 \\
		3 & 1.796E-04 &	4.831E-06 &	4.802E-06 &	2.000E+00 &	1.993E+00 &	1.991E+00 \\
		4 & 4.505E-05 &	1.212E-06 &	1.204E-06 &	1.995E+00 &	1.995E+00 &	1.995E+00 \\
		5 &1.148E-05 &	3.106E-07 &	3.087E-07 &	1.972E+00 &	1.964E+00 &	1.964E+00 \\
		\hline		
	\end{tabular}
	\label{error_MDG_1}
\end{table}

\begin{table} [H] \footnotesize
	\caption{Test 1. \(L_2\)-norms of errors and convergence orders in space: MDG, BCs scenario 2.}
	\centering
	\begin{tabular}{c c c c c c c}
		ref. level & \(L_2^{\Psi}\) & $L_2^{u_x}$ & $L_2^{u_y}$ & \(r_h^{\Psi}\) & $r_h^{u_x}$ & $r_h^{u_y}$ \\
		\hline
		0 & 1.188E-02&	2.402E-04 &	2.546E-04 & & & \\	
		1 & 2.981E-03 &	6.137E-05 &	6.518E-05 &	1.995E+00 &	1.970E+00 &	1.966E+00 \\
		2& 7.457E-04 &	1.547E-05 &	1.631E-05 &	1.999E+00 &	1.992E+00 &	1.998E+00 \\
		3& 1.861E-04 &	3.867E-06 &	4.092E-06 &	1.999E+00 &	1.994E+00 &	1.994E+00 	\\
		4& 4.659E-05 &	9.662E-07 &	1.027E-06 &	2.000E+00 &	2.000E+00 &	1.993E+00 \\
		5 &1.163E-05 &	2.476E-07 &	2.634E-07 &	1.999E+00 &	1.964E+00 &	1.963E+00 \\
		\hline		
	\end{tabular}
	\label{error_MDG_2}
\end{table}

\begin{table} [H] \footnotesize
	\caption{Test 1. \(L_2\)-norms of errors and convergence orders in space: BDG, BCs scenario 1.}
	\centering
	\begin{tabular}{c c c c c c c}
		ref. level & \(L_2^{\Psi}\) & $L_2^{u_x}$ & $L_2^{u_y}$ & \(r_h^{\Psi}\) & $r_h^{u_x}$ & $r_h^{u_y}$   \\
		\hline
		0 & 1.758E+01 &	1.955E-01 &	1.939E-01 & & & \\	
		1 & 4.423E+00 &	2.126E-02 &	2.138E-02 &	1.990E+00 &	3.201E+00 &	3.181E+00 \\
		2 & 1.107E+00 &	3.123E-03 &	3.122E-03 &	1.999E+00 &	2.767E+00 &	2.775E+00 \\
		3 & 2.770E-01 &	8.013E-04 &	8.017E-04 &	1.998E+00 &	1.963E+00 &	1.962E+00 \\
		4 & 6.544E-02 &	2.089E-04 &	2.098E-04 &	2.082E+00 &	1.939E+00 &	1.934E+00 \\
		5 & 1.640E-02 &	5.586E-05 &	5.617E-05 &	1.997E+00 &	1.903E+00 &	1.901E+00 \\				
		\hline		
	\end{tabular}
	\label{error_BDG_1}
\end{table}

\begin{table} [H] \footnotesize
	\caption{Test 1. \(L_2\)-norms of errors and convergence orders in space: BDG, BCs scenario 2.}
	\centering
	\begin{tabular}{c c c c c c c}
		ref. level & \(L_2^{\Psi}\) & $L_2^{u_x}$ & $L_2^{u_y}$ & \(r_h^{\Psi}\) & $r_h^{u_x}$ & $r_h^{u_y}$   \\
		\hline
		0 & 2.389E+01 &	2.021E-01 &	1.976E-01 & & & \\			
		1 & 6.709E+00 &	2.657E-02 &	2.609E-02 &	1.832E+00 &	2.927E+00 &	2.921E+00 \\
		2 & 1.489E+00 &	3.604E-03 &	3.477E-03 &	2.171E+00 &	2.882E+00 &	2.908E+00 \\
		3 & 3.765E-01 &	9.022E-04 &	8.801E-04 &	1.984E+00 &	1.998E+00 &	1.982E+00 \\
		4 & 9.428E-02 &	2.350E-04 &	2.291E-04 &	1.998E+00 &	1.941E+00 &	1.941E+00 \\
		5 & 2.363E-02 &	6.251E-05 &	6.099E-05 &	1.996E+00 &	1.910E+00 &	1.910E+00 \\
		\hline		
	\end{tabular}
	\label{error_BDG_2}
\end{table}

\begin{table} [H] \footnotesize
	\caption{Test 1. \(L_2\)-norms of errors and convergence order in time: MDG, BCs scenario 1.}
	\centering
	\begin{tabular}{c c c c c c c}
		\(\Delta t\) & \(L_2^\Psi\) & $L_2^{u_x}$ & $L_2^{u_y}$ & \(r_{\Delta t}^\Psi\) & $r_{\Delta t}^{u_x}$ & $r_{\Delta t}^{u_y}$  \\
		\hline
		1.E-01 & 1.117E-02 & 3.081E-04 &	3.017E-04 & & &			\\
		1.E-02 & 1.130E-03 & 3.088E-05 &	3.053E-05 &	9.95E-01 &	9.99E-01 &	9.95E-01 \\
		1.E-03 & 1.140E-04 & 3.095E-06 &	3.053E-06 &	9.96E-01 &	9.99E-01 &	1.00E+00 \\
		1.E-04 & 1.148E-05 & 3.106E-07 &	3.087E-07 &	9.97E-01 &	9.98E-01 &	9.95E-01 \\		
		\hline		
	\end{tabular}
	\label{error_time_MDG_1}
\end{table}

\begin{table} [H] \footnotesize
	\caption{Test 1. \(L_2\)-norms of errors and convergence orders in time: MDG, BCs scenario 2.}
	\centering
	\begin{tabular}{c c c c c c c}
          \(\Delta t\) & \(L_2^\Psi\) & $L_2^{u_x}$ & $L_2^{u_y}$ & \(r_{\Delta t}^\Psi\) & $r_{\Delta t}^{u_x}$ & $r_{\Delta t}^{u_y}$  \\
		\hline
		1.E-01 &	1.156E-02 &	2.436E-04 &	2.603E-04& & & \\
		1.E-02 &	1.158E-03 &	2.436E-05 &	2.609E-05 &	9.991E-01 &	1.000E+00 &	9.990E-01 \\
		1.E-03 &	1.162E-04 &	2.460E-06 &	2.603E-06 &	9.986E-01 &	9.956E-01 &	1.001E+00 \\
		1.E-04 &	1.165E-05 &	2.477E-07 &	2.634E-07 &	9.988E-01 &	9.971E-01 &	9.948E-01 \\				
		\hline		
	\end{tabular}
	\label{error_time_MDG_2}
\end{table}

\begin{table} [H] \footnotesize
	\caption{Test 1. \(L_2\)-norms of errors and convergence orders in time: BDG, BCs scenario 1.}
	\centering
	\begin{tabular}{c c c c c c c}
	  \hline
\(\Delta t\) & \(L_2^\Psi\) & $L_2^{u_x}$ & $L_2^{u_y}$ & \(r_{\Delta t}^\Psi\) & $r_{\Delta t}^{u_x}$ & $r_{\Delta t}^{u_y}$  \\
		\hline
		1.E-01 &	1.588E+01&	5.444E-02&	5.461E-02 & & & \\			
		1.E-02 &	1.591E+00&	5.456E-03&	5.471E-03&	9.991E-01&	9.990E-01&	9.992E-01\\
		1.E-03 &	1.612E-01&	5.524E-04&	5.544E-04&	9.943E-01&	9.947E-01&	9.942E-01\\
		1.E-04 &	1.640E-02&	5.586E-05&	5.617E-05&	9.926E-01&	9.952E-01&	9.943E-01\\		
		\hline		
	\end{tabular}
	\label{error_time_BDG_1}
\end{table}

\begin{table} [H] \footnotesize
	\caption{Test 1. \(L_2\)-norms of errors and convergence orders in time: BDG, BCs scenario 2.}
	\centering
	\begin{tabular}{c c c c c c c}
\(\Delta t\) & \(L_2^\Psi\) & $L_2^{u_x}$ & $L_2^{u_y}$ & \(r_{\Delta t}^\Psi\) & $r_{\Delta t}^{u_x}$ & $r_{\Delta t}^{u_y}$  \\
		\hline
		1.E-01 &	2.288E+01&	6.106E-02&	5.938E-02 & & & \\			
		1.E-02&	2.294E+00 &	6.119E-03&	5.947E-03 &	9.989E-01 &	9.991E-01 &	9.993E-01 \\
		1.E-03&	2.323E-01&	6.187E-04&	6.025E-04 &	9.944E-01 &	9.951E-01 &	9.943E-01 \\
		1.E-04&	2.363E-02&	6.251E-05&	6.099E-05 &	9.926E-01 &	9.956E-01 &	9.947E-01 \\    	
		\hline		
	\end{tabular}
	\label{error_time_BDG_2}
\end{table}

\subsection{Test 2: lid driven cavity}
Stokes flow in closed rectangular cavities is a useful approximation for many applications, such as ceramics casting, polymer processing, roll coating, etc., see, e.g., \cite{MIKHAYLENKO2018103, GUMGUM201765}. We consider the classical lid-driven cavity test case in a square cavity \(\left[0, 1\right] \times \left[0, 1\right]\), as sketched in \cref{fig_test2_sketch}. The initial conditions (ICs) are zero velocity and pressure, while the BCs are assigned as in \cref{fig_test2_sketch}, where the upper wall moves along the \(x-\)direction with assigned velocity \(u_0\) and \(s\) is the horizontal velocity component assigned at the bottom wall. We consider the cases \(s = u_0\), \(s = 0\), and \(s = -u_0\). The Reynolds number \(Re=\frac{u_0 \ L}{\nu}\) is set to 1, where \(L\) is the size of the cavity. The computational grids are the same mildly and badly distorted ones (MDG and BDG) used in Test 1, presented in \cref{test1}. A pressure BC $\Psi = 0$ is assumed at the bottom-left corner, to fix the pressure values. \par

We present dimensionless results, dividing the spatial \(x\) and \(y\) coordinates by the unitary cavity length, the velocity components and velocity magnitude by \(u_0\), and the kinematic pressure by \(u_0^2\). \par

In \cref{fig_test2_plots} we plot the solution computed on the 2\(^{nd}\) refinement grid level of the BDG assuming \(s=0\) (fixed bottom wall). Very similar results have been obtained with the MDG and for brevity are not shown. The results are qualitatively in very good agreement with the literature, see \cite[Figure 10 (left)]{ARAYA201729} and \cite[Figure 3.8 a,b]{MIKHAYLENKO2018103}. In \cite{ARAYA201729}, the authors use a multiscale hybrid-mixed method, while in \cite{MIKHAYLENKO2018103}, the author presents a regularized solution of the Stokes equations, where a cut-off blob function replaces the classical delta-Dirac one. In \cref{fig_test2data} we plot the computed profiles with both MDG and BDG of the velocity components and pressure along the vertical and horizontal mid-lines of the domain, in the case of \(s=0\). The grid distortion does not seem to affect the results. The \(u_x\) profiles along the vertical mid-line match very well the solution provided in \cite{MIKHAYLENKO2018103}. \par

The velocity streamlines, computed over the \(2^{nd}\) refinement of the BDG, are in very good agreement with the solution provided in \cite{GUMGUM201765}, for the cases of \(s=u_x\) and \(s=-u_x\),  as shown in \cref{fig_test2_plot_velo_bottom}. In \cite{GUMGUM201765}, the authors solve the steady part of the Stokes flows problem using boundary element methods, while the time derivative is approximated by the finite difference method. The streamlines on the MDG and finer grids are very similar, and for brevity are not shown. Both velocity components computed along the vertical and horizontal mid-lines in the cases of \(s=u_x\) and \(s=-u_x\) fit very well the results provided in \cite{GUMGUM201765}, see \cref{fig_test2_data_bottom_velo} and \cref{fig_test2_data_bottom_reverse_velo}.

Overall, we conclude that the results from the presented method for this benchmark problem match very well literature solutions. Moreover the results are robust with respect to grid distortion.

\begin{figure}
	\centering
	\includegraphics[width=0.3\textwidth]{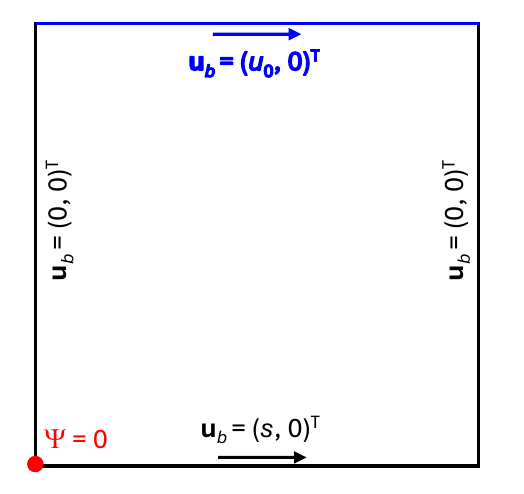}
	\caption{Test 2. Sketch of the domain and assigned BCs.}
	\label{fig_test2_sketch}
\end{figure}

\begin{figure}
	\centering
	\includegraphics[width=1\textwidth]{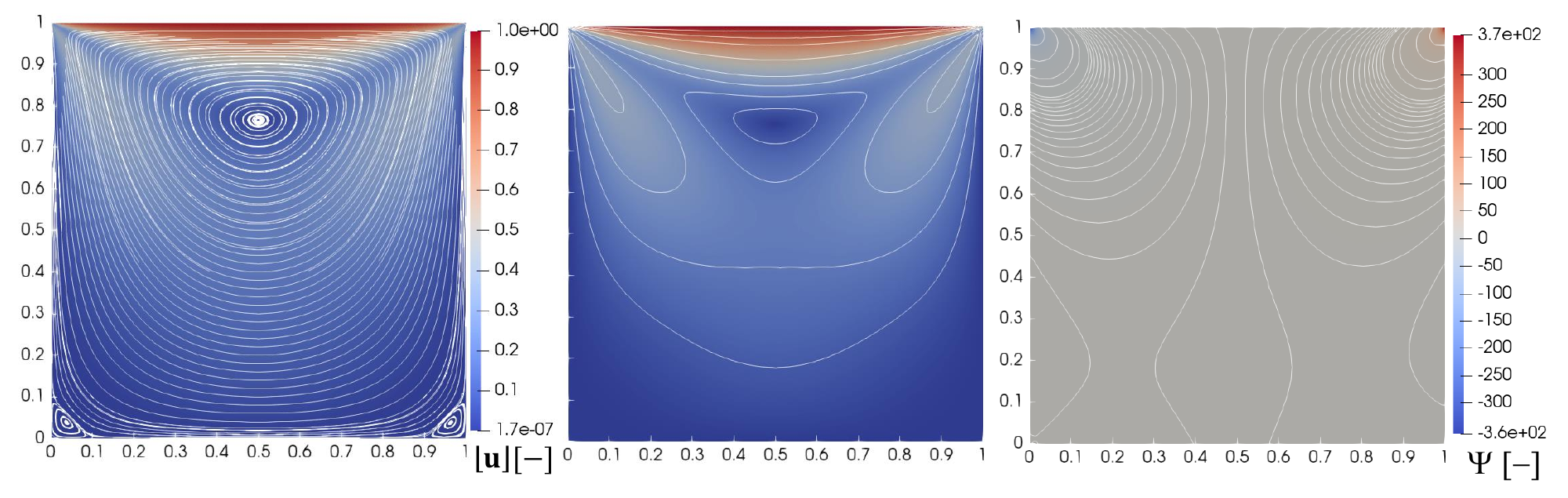}
	\caption{Test 2. Computed solution on the \(2^{nd}\) refinement level of the BDG. Left: velocity field. Center: streamline contours. Right: pressure field.}
	\label{fig_test2_plots}
\end{figure}

\begin{figure}
\centering
\includegraphics[width=0.8\textwidth]{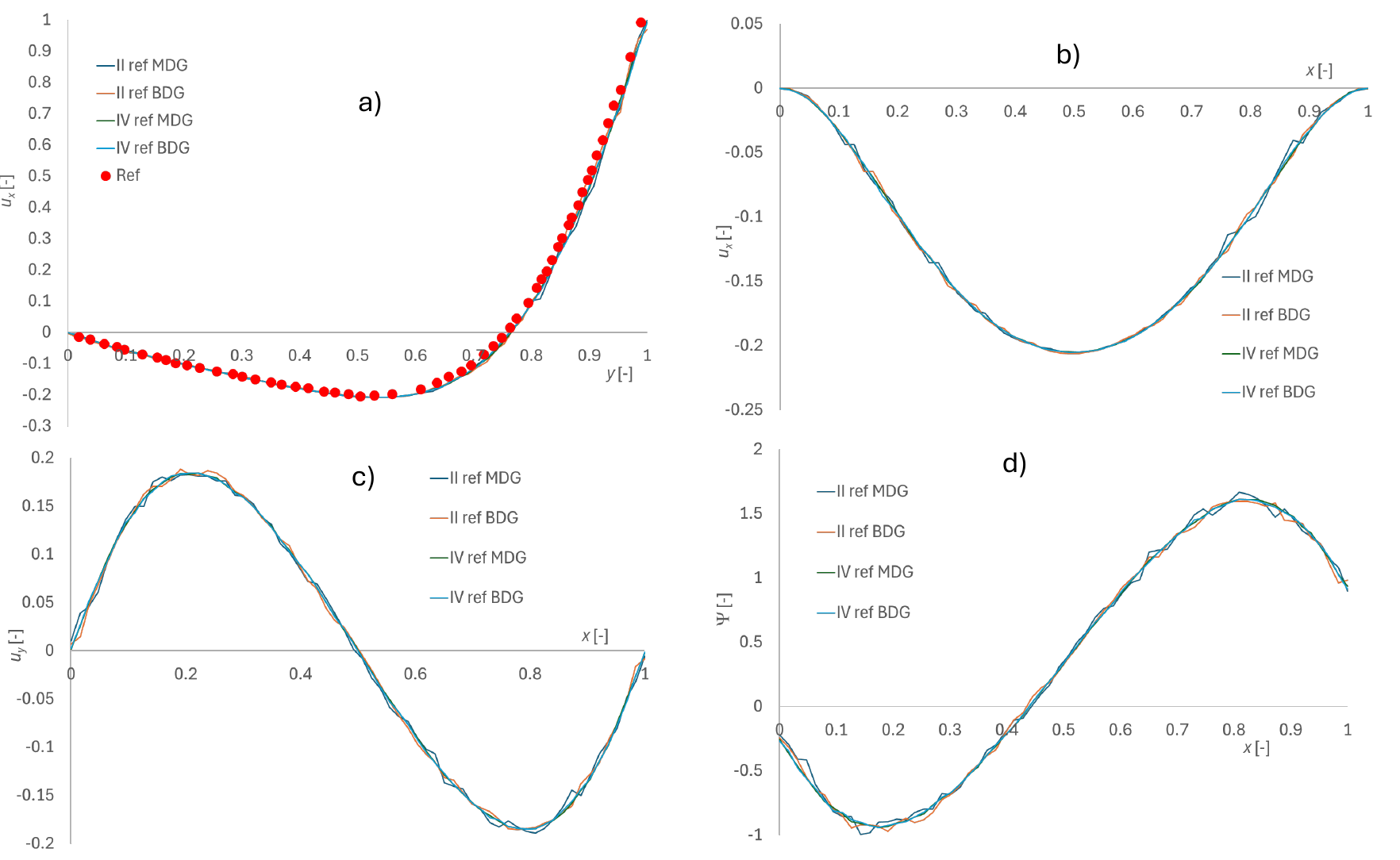}
\caption{Test 2. Computed solution on the \(2^{nd}\) and \(4^{th}\) refinements of the MDG and BDG: a) profile of \(u_x\) along the vertical mid-line, b) profile of \(u_x\) along the horizontal mid-line, c) profile of \(u_y\) along the horizontal mid-line, d) profile of \(\Psi\) along the horizontal mid-line. ``Ref.'' marks the solution provided in \cite{MIKHAYLENKO2018103}.}
	\label{fig_test2data}
\end{figure}

\begin{figure}
	\centering
	\includegraphics[width=0.75\textwidth]{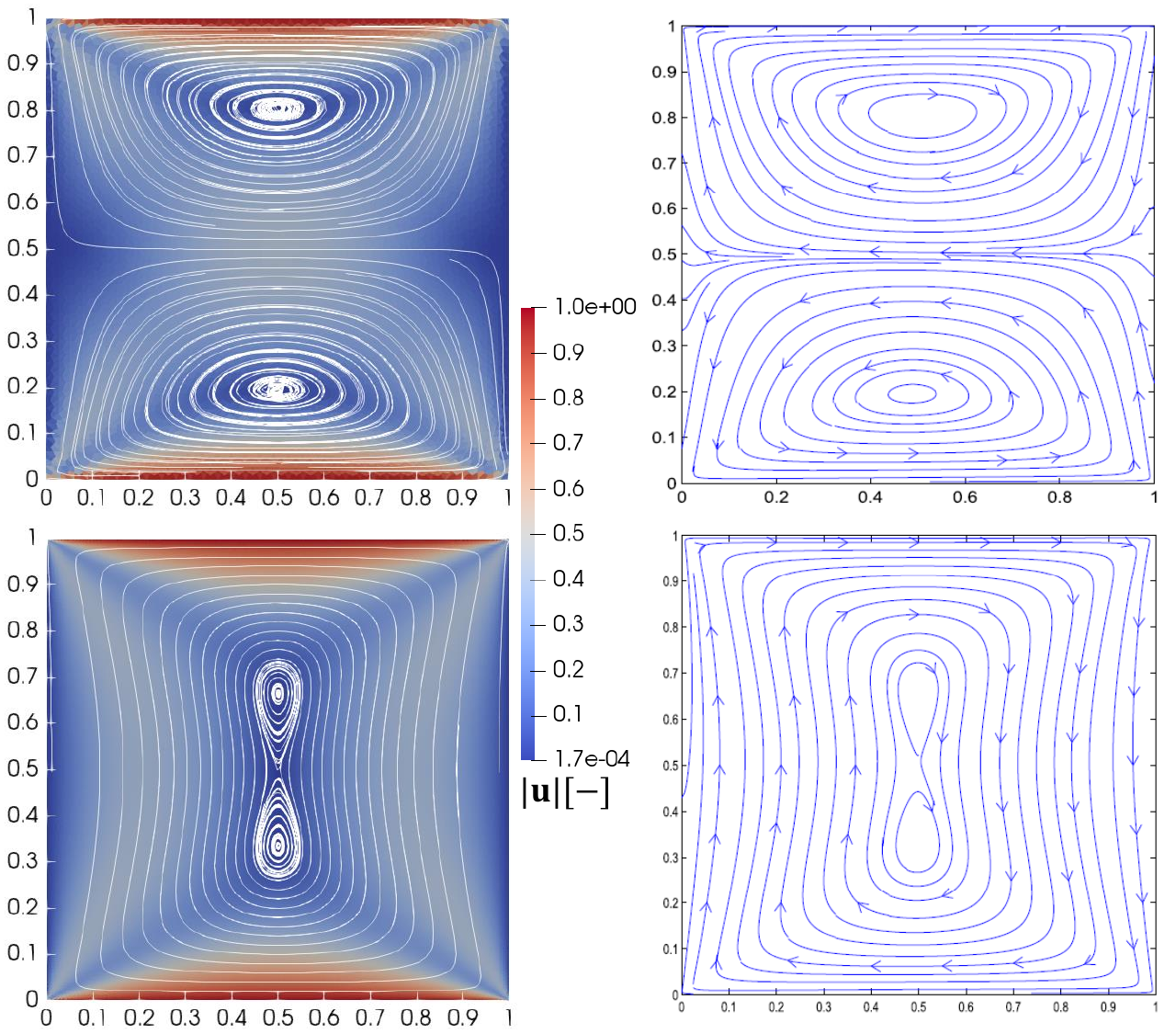}
	\caption{Test 2. Computed velocity streamlines on the \(2^{nd}\) refinement of BDG. Top row: case \(s=u_{0,x}\); left: our solution, right: literature results in \cite{GUMGUM201765}. Bottom row: case \(s=-u_{0,x}\); left: our solution, right: literature results in \cite{GUMGUM201765}. }
	\label{fig_test2_plot_velo_bottom}
\end{figure}

\begin{figure}
	\centering
	\includegraphics[width=0.8\textwidth]{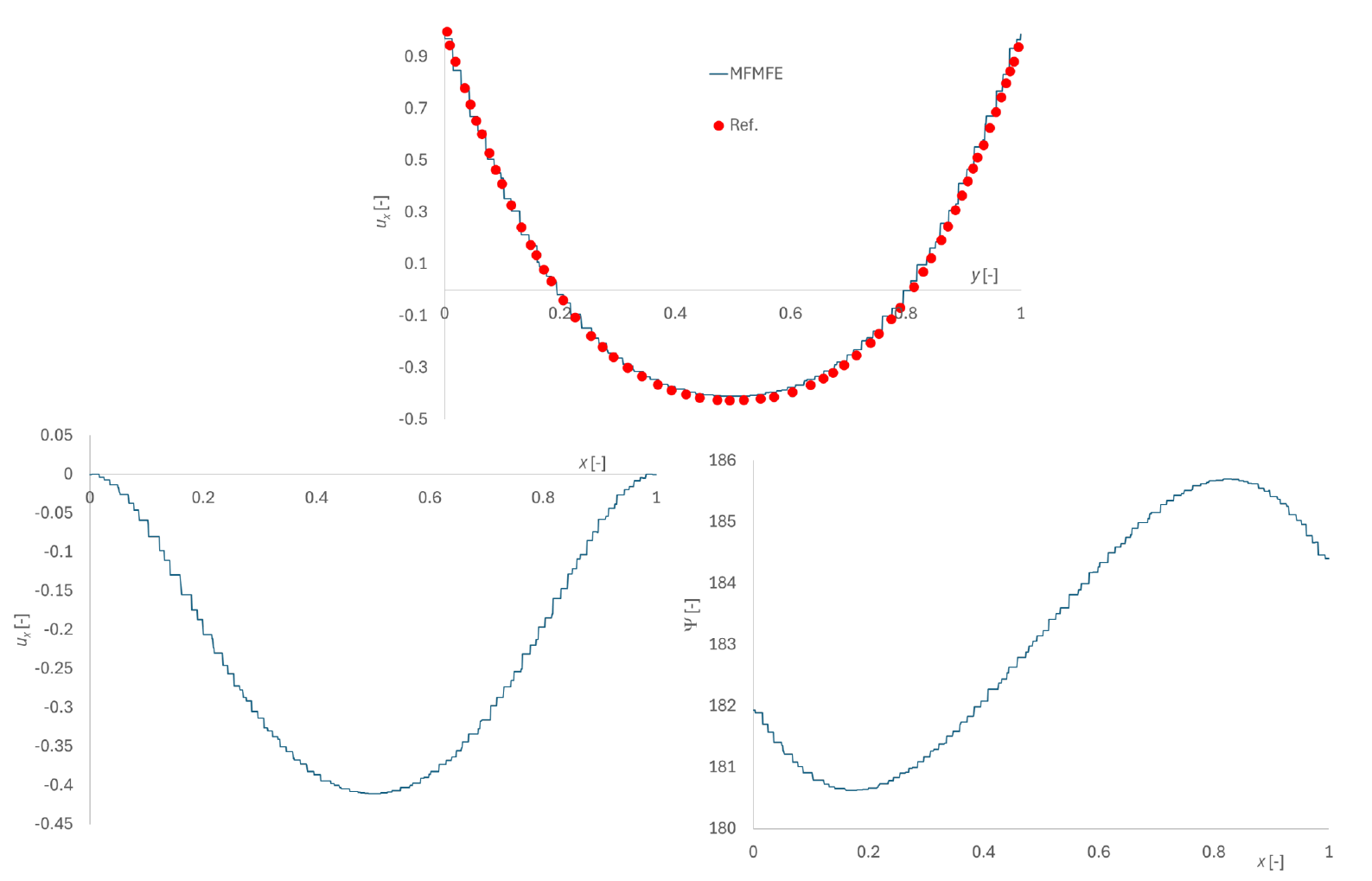}
	\caption{Test 2. Computed profiles of the velocity components and pressure on the \(2^{nd}\) refinement of BDG, case \(s=u_{0,x}\). Top: \(u_x\) along the vertical mid-line with ``Ref.'' showing the results in \cite{GUMGUM201765}. Bottom row: data along the horizontal mid-line; left: \(u_x\), right: \(\Psi\). }
	\label{fig_test2_data_bottom_velo}
\end{figure}

\begin{figure}
	\centering
	\includegraphics[width=0.8\textwidth]{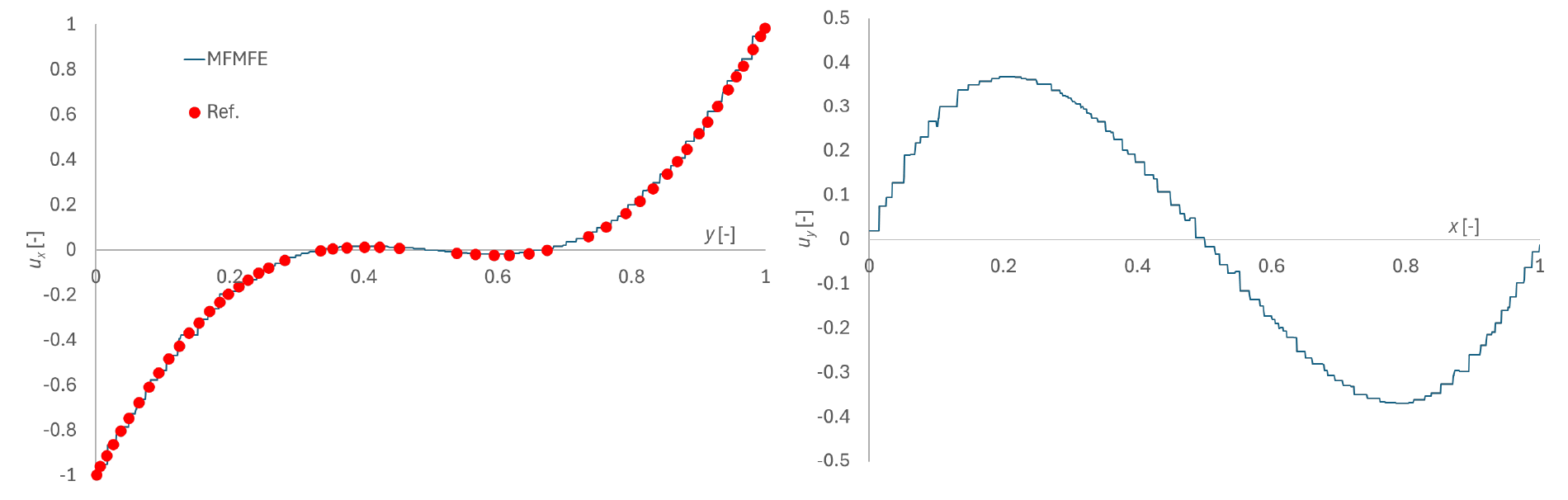}
	\caption{Test 2. Computed profiles of the velocity components and pressure over the \(2^{nd}\) refinement BDG, case \(s=-u_{0,x}\). Left: \(u_x\) along the vertical mid-line with ``Ref.'' showing the results in \cite{GUMGUM201765}. Right: \(u_y\) along the horizontal mid-line.}
	\label{fig_test2_data_bottom_reverse_velo}
\end{figure}

\subsection{Test 3: Stokes flow in a channel with smoothly changing width}

Stokes flows in confined channel with variable geometries is encountered in many industrial, biomedical, and biological systems, such as microfluidics, flow driven by contracting channel walls, lab-chip technologies, drug carrying micro-devices, etc., see \cite{doi:10.1098/rspa.2017.0234}. We consider the channel in \cref{fig_test3_geo_BCs}, with characteristic length and depth \(L_0\) and \(h_0\), respectively, and set \(\delta = \frac{h_0}{L_0}\). The ICs are zero velocity and pressure, and the BCs are assigned as in \cref{fig_test3_geo_BCs}, with parabolic Poiseuille inflow flux \(q\), no-slip BCs on the upper and lower channel walls, and zero pressure at the downstream channel end. Following \cite{doi:10.1098/rspa.2017.0234}, we introduce the dimensionless variables
\begin{equation}
X=\frac{x}{L_0}, \quad Y=\frac{y}{h_0},\quad H=\frac{h}{h_0}, \quad U_x = \frac{u_x}{q/h_0}, \quad U_y = \frac{u_y}{q/h_0}, \quad \hat{\Psi} = \frac{\Psi}{\left(q/h_0\right)^2},
\end{equation}	
and set the scale for time to \(\frac{q}{h_0^2}\). The equation for the (dimensionless) restriction of the channel width between \(X=-1\) and \(X = 1\) is \cite{doi:10.1098/rspa.2017.0234, doi:10.1137/23M1576876}
\begin{equation}
H(X) = 1 - \frac{\lambda}{2} \left(1+\cos(\pi X)\right), \quad 0 \le \lambda < 1,
\end{equation}
where \(\lambda\) is the maximum (dimensionless) width contraction. The (dimensionless) inflow Poiseuille velocity profile is
\begin{equation}
\mathbf{U}\left(Y\right)=\left(6\left(Y - Y^2\right),0\right)^T.
\end{equation}
In the presented simulations we set \(\delta = 1\) and \(\lambda=0.2, \ 0.5, \ 0.8\). The Reynolds number \(Re\), computed according to the dimensionless inflow channel depth and the maximum inflow velocity, is 1.5. We run our simulations over a set of coarse and fine grids, whose number of simplices \(N_T\) and vertices \(N\) are listed in \cref{test3_info_grids}. The dimensionless time step size is \(\frac{\Delta t}{u_x/h^2_0} = 1\e{-2} \). We also ran other simulations using smaller and larger time step sizes, but the results did not change significantly.

\begin{table}[ht] \footnotesize
\caption{Test 3. Number of triangles and vertices for the coarse and fine grids.}
\centering 
\begin{tabular}{c c c c c} 
	& \multicolumn{2}{c}{coarse grids}  & \multicolumn{2}{c}{fine grids}   \\
	\hline
	\(\lambda\) & \(N_T\) & \(N\) & \(N_T\) & \(N\) \\
	0.2 & 17306 & 9149 & 64079 & 32705 \\
	0.5 & 16257 & 8624 & 61429 & 31382 \\
	0.8 & 15648 & 8320 & 57090 & 29212 \\
	\hline
\end{tabular}
\label{test3_info_grids}
\end{table}

In \cref{fig_test3_plots} we show the (dimensionless) computed velocity and pressure fields with \(\lambda = 0.8\). In \cite{doi:10.1098/rspa.2017.0234}, the same problem is solved by the lubrication theory, considering the higher-order terms coming from the perturbation theory, accounting for values of \(\delta \simeq 1\), see also \cite{doi:10.1137/23M1576876}. The authors of \cite{doi:10.1098/rspa.2017.0234} provide an analytical expression of the pressure drop \(\Delta \hat{\Psi}\) from \(X=-1\) to \(X=1\) as function of the restriction parameter \(\lambda\), shown in \cref{fig_test3_plot_delta_press}, and confirmed by the experimental results in \cite{doi:10.1098/rspa.2017.0234}. The scatter between the analytical data given by the  lubrication theory and the experiments can be justified by the 3D flow effects not accounted in the 2D lubrication theory. These 3D effects generally tend to disappear as the value of \(\lambda\) increases and the channel becomes narrower. In the same \cref{fig_test3_plot_delta_press} we also plot the dimensionless pressure drops computed by the present method on the coarse and fine grids. Our results fit very well the reference data for the three investigated values of \(\lambda\), and the mesh size does not affect significantly the pressure drop.

\begin{figure}
\centering
\includegraphics[width=0.8\textwidth]{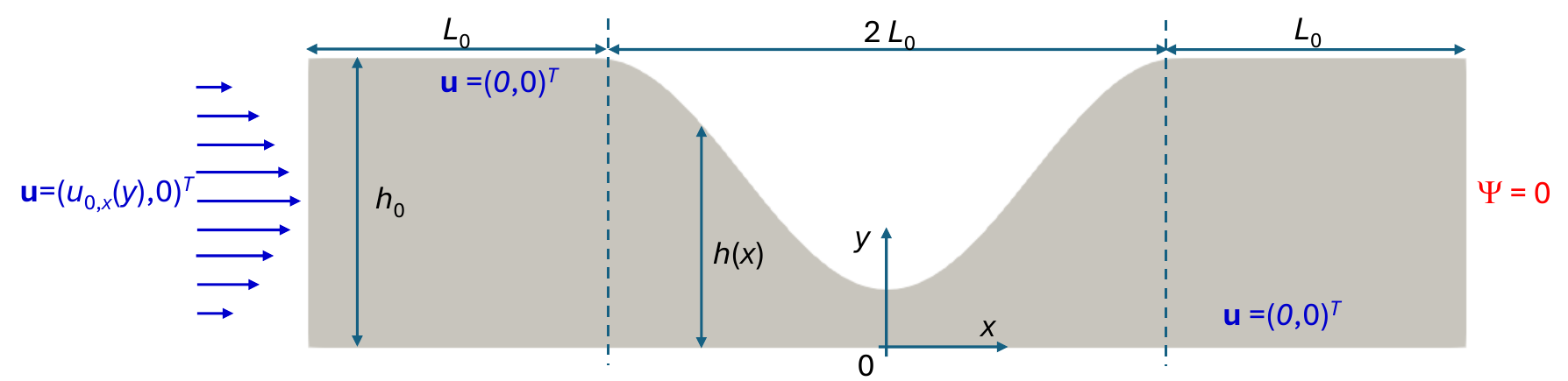}
\caption{Test 3. Channel geometry and assigned BCs.}
\label{fig_test3_geo_BCs}
\end{figure}

\begin{figure}
\centering
\includegraphics[width=0.85\textwidth]{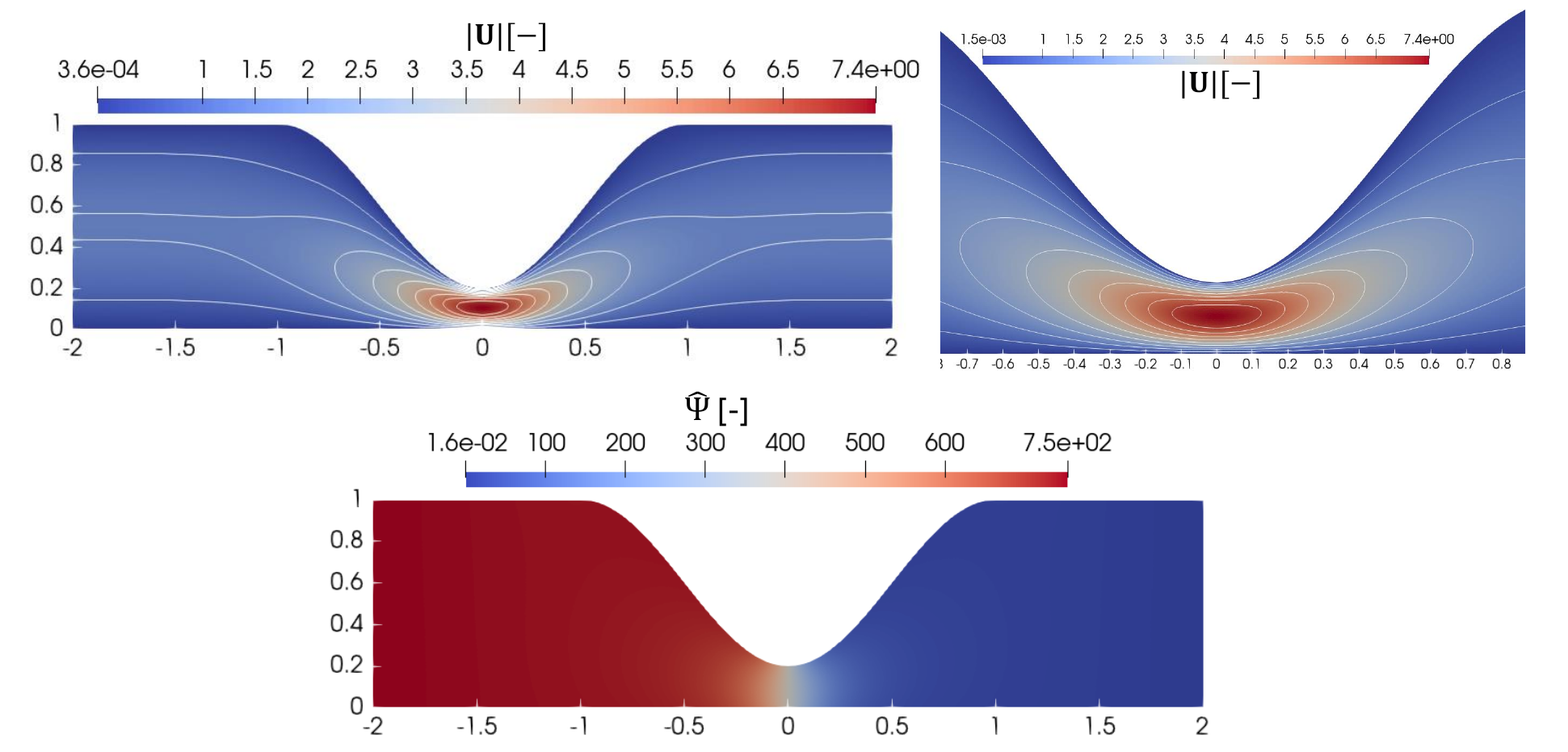}
\caption{Test 3. Computed (dimensionless) solution for \(\lambda = 0.8\). Top row: velocity field and velocity magnitude contours (left) and zoom in the restriction area (right). Bottom row: computed pressure field.}
\label{fig_test3_plots}
\end{figure}

\begin{figure}
\centering
\includegraphics[width=0.55\textwidth]{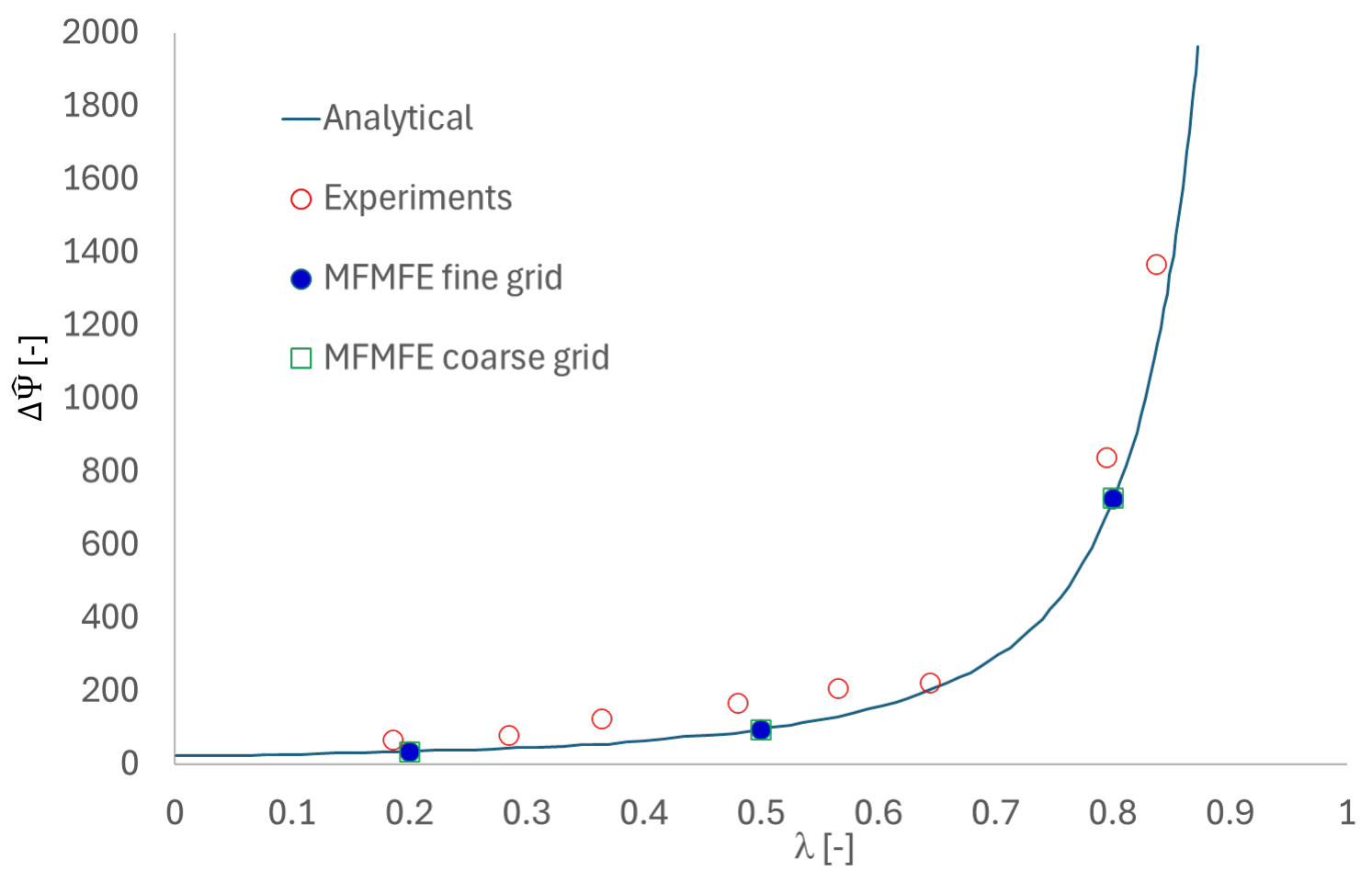}
\caption{Test 3. Computed (dimensionless) pressure drop from \(X=-1\) to \(X=1\) for different values of \(\lambda\). Nomenclature. ``Analytical'': data from the lubrication theory including the higher order terms \cite{doi:10.1098/rspa.2017.0234}. ``Experiments'': experimental data provided in \cite{doi:10.1098/rspa.2017.0234}. ``MFMFE fine grids'' and ``MFMFE coarse grids'': results from the present method.}
\label{fig_test3_plot_delta_press}
\end{figure}

\subsection{Test 4: pulsatile Stokes flow within a highly irregular domain} 

The purpose of the test in this section is to investigate the capability of the proposed method to simulate pulsatile Stokes flows within very irregular domains, which find several industrial and biomedical applications, e.g., hydrodynamics in micro and mini coiling systems, vascular hemodynamics in small vessels, and peristaltic flows. We consider flow in a duct with a central elliptic expansion and a wavy-shaped inner solid obstacle with mean radius \(R_i\), see 
the geometry in \cref{fig_test4_geo_BCs}, where the equation of the elliptic-shape expansion of the duct is also shown. We assign a pulsatile inflow Poiseuille velocity profile at the upstream end, pressure condition \(\Psi = 0\) at the downstream end, and zero velocity on the remaining boundaries of the external duct and on the surface of the inner obstacle. We non-dimentionalize the problem considering the scales for lengths (\(L_0\)), velocity (\(U_0\)), kinematic pressure (\(\Psi_0\)) and time (\(T_0\)),
\begin{equation}
L_0 = D = 2 \ R, \quad U_0 = \hat{U}_{max}, \quad \Psi_0 = U_{max}^2, \quad T_0 = \frac{D}{ U_{max}},
\end{equation}
where \(\hat{U}_{max}\) is the maximum value of the Poiseuille assigned inflow velocity profile and \(D = 2 \ R\) is the diameter of the duct at the upstream end. The significant geometrical and kinematic variables are listed in \cref{table_test4_geo}. The wavy profile of the inner obstacle is similar to the one used in test 4 in \cite{Arico-thermal},
\begin{equation}
x=r_i \sin \phi, \quad y=r_i \cos \phi, \quad r_i = R_i + a \cos n\phi, \quad 0 \le \phi \le 2 \pi,
\end{equation} 
with \(a\) and \(n\) listed in \cref{table_test4_geo}. The (dimensionless) inflow pulsatile Poiseuille velocity is
\begin{equation}
u_x\left(r,t\right) = U_{max}\left(t\right) \frac{\left(R^2 - d_r^2\right)}{R^2}, \quad U_{max}\left(t\right)= \hat{U}_{max} \cos \left(\frac{2}{t_0} \pi t\right), \quad u_y=0, \quad \forall t \ge 0,
\end{equation}
where \(d_r\) is the distance of any point \(r\) of the inflow section from its center, \(t_0\) is the period of the pulsation, and the other symbols are specified in \cref{fig_test4_geo_BCs}. \par

\begin{figure}
\centering
\includegraphics[width=.9\textwidth]{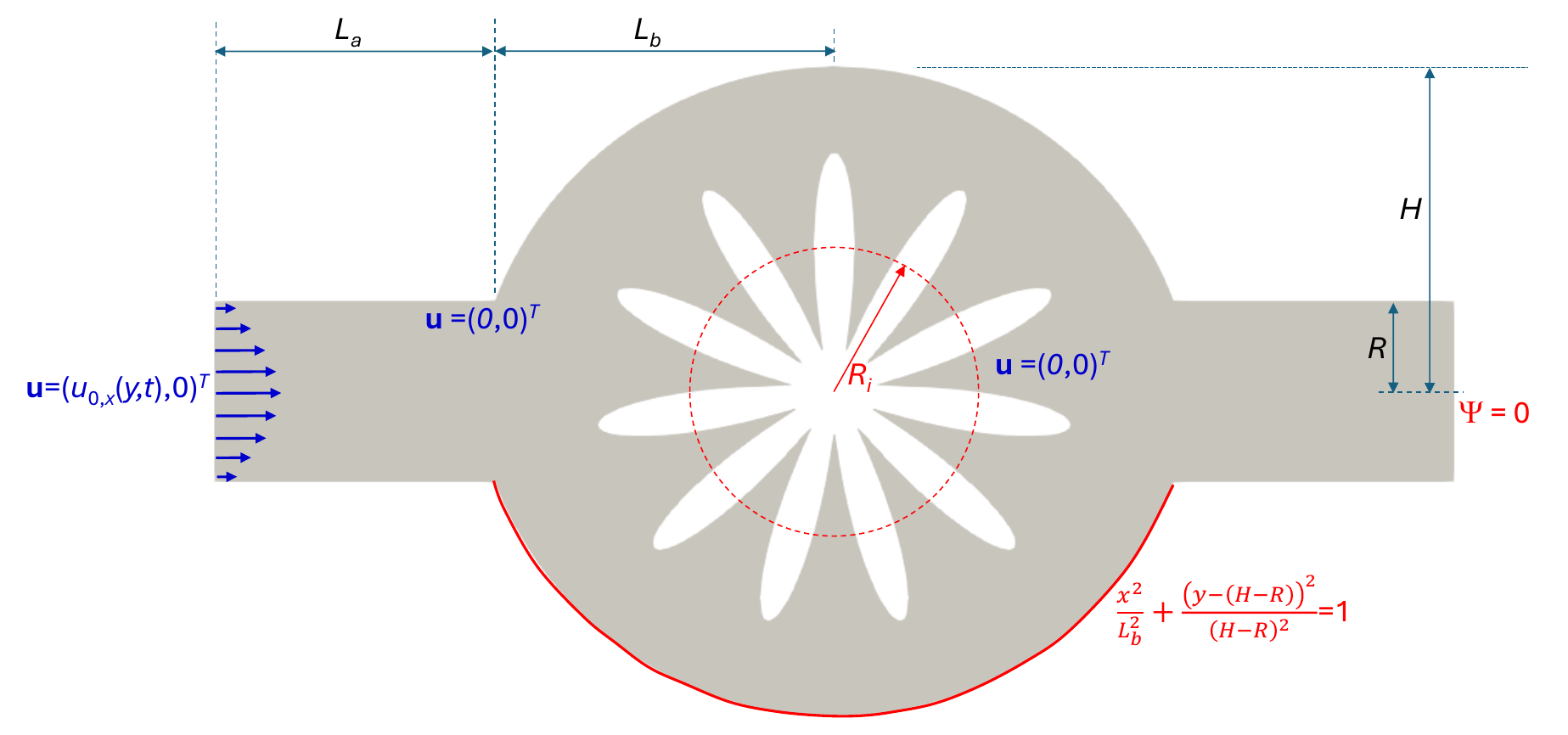}
\caption{Test 4. Channel geometry and assigned BCs. Values of the variables are listed in \cref{table_test4_geo}.}
\label{fig_test4_geo_BCs}
\end{figure}

\begin{table}
  \footnotesize
\caption{Test 4. Values of the significant geometrical, kinematic and rheological variables.}
\centering 
\begin{tabular}{c  c} 
	\(R\) [m] &  0.004  \\
	\hline
	\(\phi_0\) [rad] & \(\pi/12\) \\
	\hline
	\(L_a\) & \(R \ \pi\)   \\
	\hline
	$L_B$  & $2 \ R \ \tan \phi_0$  \\
	\hline
	$H$ & $\frac{R}{\sin \phi_0}$ \\
	\hline
	$R_i$ & \(\frac{R}{2.5 \ D^2 \ \sin \phi_0}\) \\
	\hline
	\( a\) [-] & 0.7 \\
	\hline
	\(n\) [-] & 11 \\
	\hline
	\(\hat{U}_{max}\) [m/s] & 0.005  \\
	\hline
	\(\nu\) [m\(^2\)/s] & \(6.7\e{-6}\)  \\
\end{tabular}
\label{table_test4_geo}
\end{table}

We simulate three scenarios using the values of \(t_0\) listed in \cref{table_test4_pulse}. In the same table we list the associated Womersley number \(\alpha = R\sqrt{\frac{1}{t_0 \nu}}\) \cite{https://doi.org/10.1113/jphysiol.1955.sp005276}. The computational grid has 59949 triangles and 31831 vertices, and in \cref{fig_test4_zoom_grid} we plot a zoom of the grid discretizing the inner portions of the obstacle. The (dimensionless) grid size ranges between 0.05 and 0.0015, and the (dimensionless) time step size is \(\Delta t = \frac{t_0}{T_0 \cdot 80}\). \par

In \cref{fig_test4_velo_comps_all_womersley} we plot the radial profiles of the (dimensionless) velocity components close to the inflow section. As expected, the horizontal velocity becomes flatter as \(\alpha\) increases \cite{https://doi.org/10.1113/jphysiol.1955.sp005276}. \cref{fig_test4_plot_wom_1} shows the velocity and pressure fields every quarter of period in the case of \(\alpha = 1.09\), and \cref{fig_test4_plot_wom_1-2} shows the zooms of the velocity field close to the inner invested obstacle at the same simulation times. Complex vortical structures, composed of four or more vortices with alternating flow direction, can be observed between the waves of obstacle perimeter. \par

We compute the forces of the fluid on the perimeter of the obstacle. The total force \(\mathbf{F}\) is the sum of the kinematic pressure force, \(\mathbf{F}_{\Psi}\), and the viscous force, \(\mathbf{F}_{\nu}\). We divide the perimeter in \(N_p\) parts (edges) and set
\begin{equation}
\mathbf{F}= \mathbf{F}_{\Psi} +\mathbf{F}_{\nu} = \sum_{i=1}^{N_p} \left(\mathbf{F}_{\Psi,i} + \mathbf{F}_{\nu,i}\right).
\end{equation}
The force \(\mathbf{F}_{\Psi}\) on the \(i\)-th edge of the obstacle perimeter is computed as
\begin{equation}
\mathbf{F}_{\Psi} = |e_i| \frac{\Psi_1 + \Psi_2}{2} \mathbf{n}_i
\end{equation}
where \(|e_i|\) is the length of the \(i\)-th edge, \(\Psi_1\) and \(\Psi_2\) are the kinematic pressure values computed at the two vertices of the edge \(e_i\), and \(\mathbf{n}_i\) is the unit normal vector to \(e_i\), pointing outward from the fluid region. The two values \(\Psi_1\) and \(\Psi_2\) are computed according to the \(P_1\) approximation of \(\Psi\) within the triangle \(E\) sharing edge \(e_i\) with the perimeter of the solid obstacle. The \(x\)- and \(y\)-components of the viscous force acting on \(e_i\) are computed as
\begin{equation}
F_{x,\nu,i} = \nu \nabla u_x \cdot \mathbf{n}_i |e_i|, \qquad F_{y,\nu,i} = \nu \nabla u_y \cdot \mathbf{n}_i |e_i|,
\end{equation}
where \( \nabla u_x\) and \( \nabla u_y\) are the gradient vectors of the velocity components computed within the triangle \(E\) sharing edge \(e_i\) with the obstacle. These gradient vectors are computed according to the linear variation of the velocity components within \(E\). \par

\cref{fig_test4_plot_wom_1-3,fig_test4_plot_wom_1-4} show the vectors \(\mathbf{F}\) and \(\mathbf{F}_{\nu}\), respectively, computed for \(\alpha = 1.09\) at every quarter-period. \(|\mathbf{F}_{\nu}|\) is approximately one magnitude order smaller than \(|\mathbf{F}|\). Observe the opposite direction of the vectors at time \(0\) and \(t_0/2\), as well as at times \(t_0/4\) and \(3 \ t_0/4\), respectively, due to the periodicity of the process. The changes of the direction of \(\mathbf{F}_{\nu}\) on the boundary of the obstacle at any investigated time are due to the vortical structures within the waves of the perimeter, as discussed before. In the cases of \(\alpha = 2.19\) and \(\alpha = 4.89\), the directions of the forces are similar to those in \cref{fig_test4_plot_wom_1-3,fig_test4_plot_wom_1-4} and for brevity are not shown. \(|\mathbf{F}|\) in the cases of \(\alpha = 2.19\) and \(\alpha = 4.89\) is approximately in the range \(\left[1.65,3.375\right] \) and \(\left[4.6,15\right] \), respectively, relative to \(|\mathbf{F}|\) in the case of \(\alpha = 1.09\). \(|\mathbf{F}_{\nu}|\) in the cases of \(\alpha = 2.19\) and \(\alpha = 4.89\) is approximately in the range \(\left[1.2,2.3\right] \) and \(\left[2,5.17\right] \), respectively, relative to \(|\mathbf{F}_{\nu}|\) in the case of \(\alpha = 1.09\). 

Overall the presented results illustrate the ability of our method to simulate flows in highlly complex geometries in a numerically stable and accurate manner.

\begin{table}
  \footnotesize
\caption{Test 4. Period of the pulsation \(t_0\) of the assigned inflow velocity Poiseuille profiles and associated Womersley number \(\alpha\).}
\centering 
\begin{tabular}{c|c}
	\(t_0\) [s] & \(\alpha\)  \\
	\hline
	2 & 1.09 \\
	\hline
	0.5 & 2.19 \\
	\hline
	0.1 & 4.89 \\
\end{tabular}
\label{table_test4_pulse}
\end{table}

\begin{figure}
\centering
\includegraphics[width=0.7\textwidth]{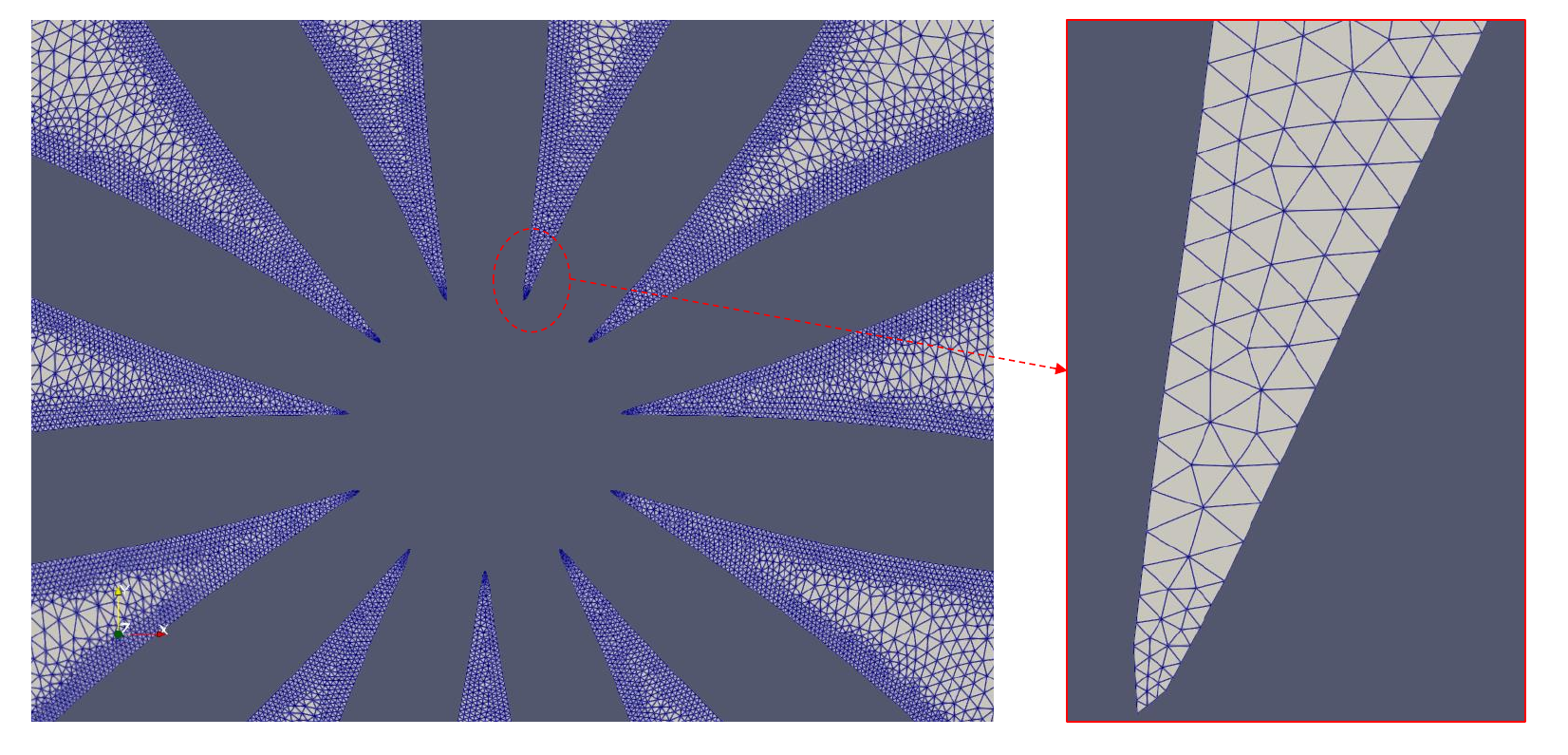}
\caption{Test 4. Zoom of the computational grid.}
\label{fig_test4_zoom_grid}
\end{figure}

\begin{figure}
\centering
\includegraphics[width=1\textwidth]{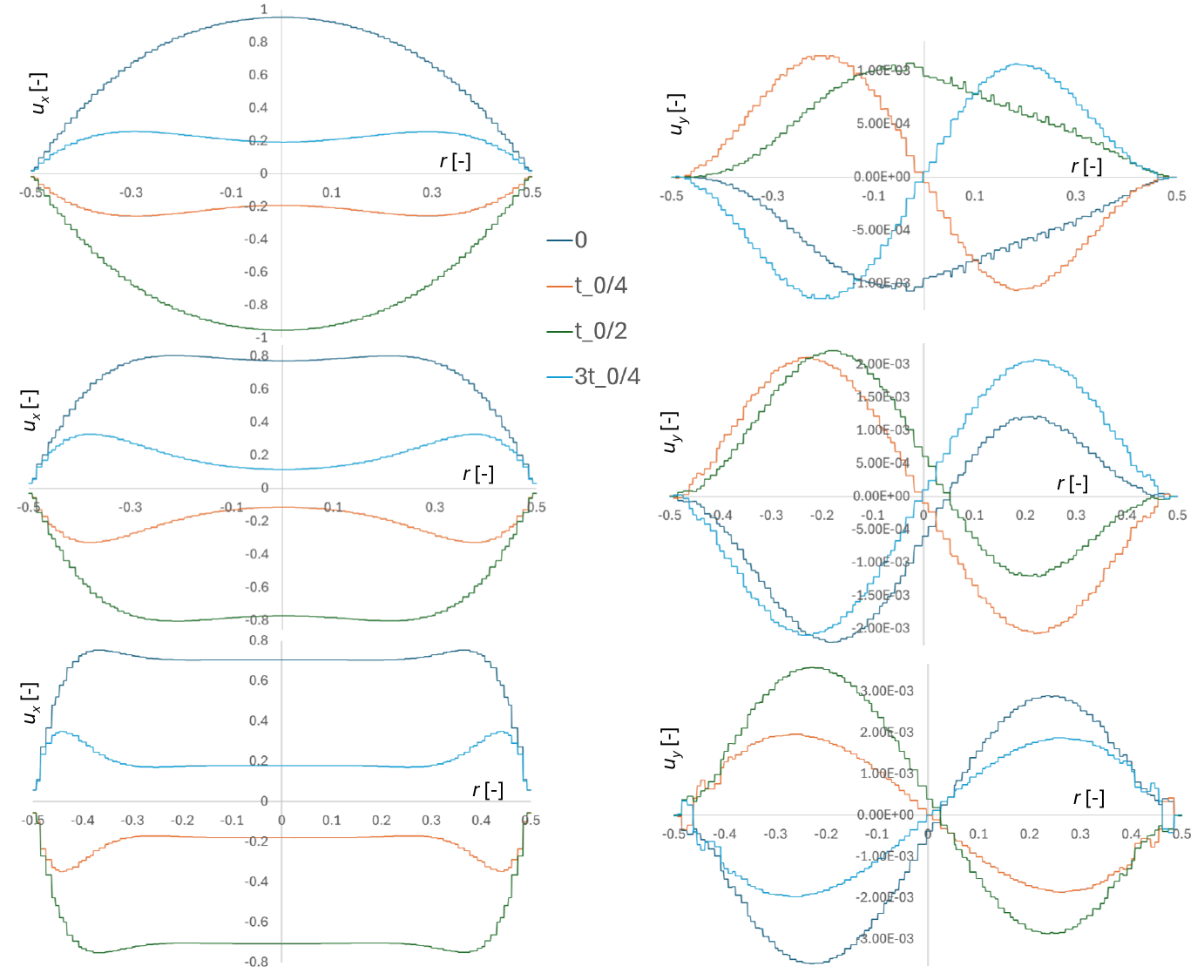}
\caption{Test 4. Radial profiles of the (dimensionless) \(u_x\) and \(u_y\) close to the inflow section. Top row: \(\alpha = 1.09\). Middle row: \(\alpha = 2.19 \). Bottom row: \(\alpha = 4.89\). Nomenclature: \(r\) is the radial distance from the mid-line of the duct section.}
\label{fig_test4_velo_comps_all_womersley}
\end{figure}

\begin{figure}
\centering
\includegraphics[width=1\textwidth]{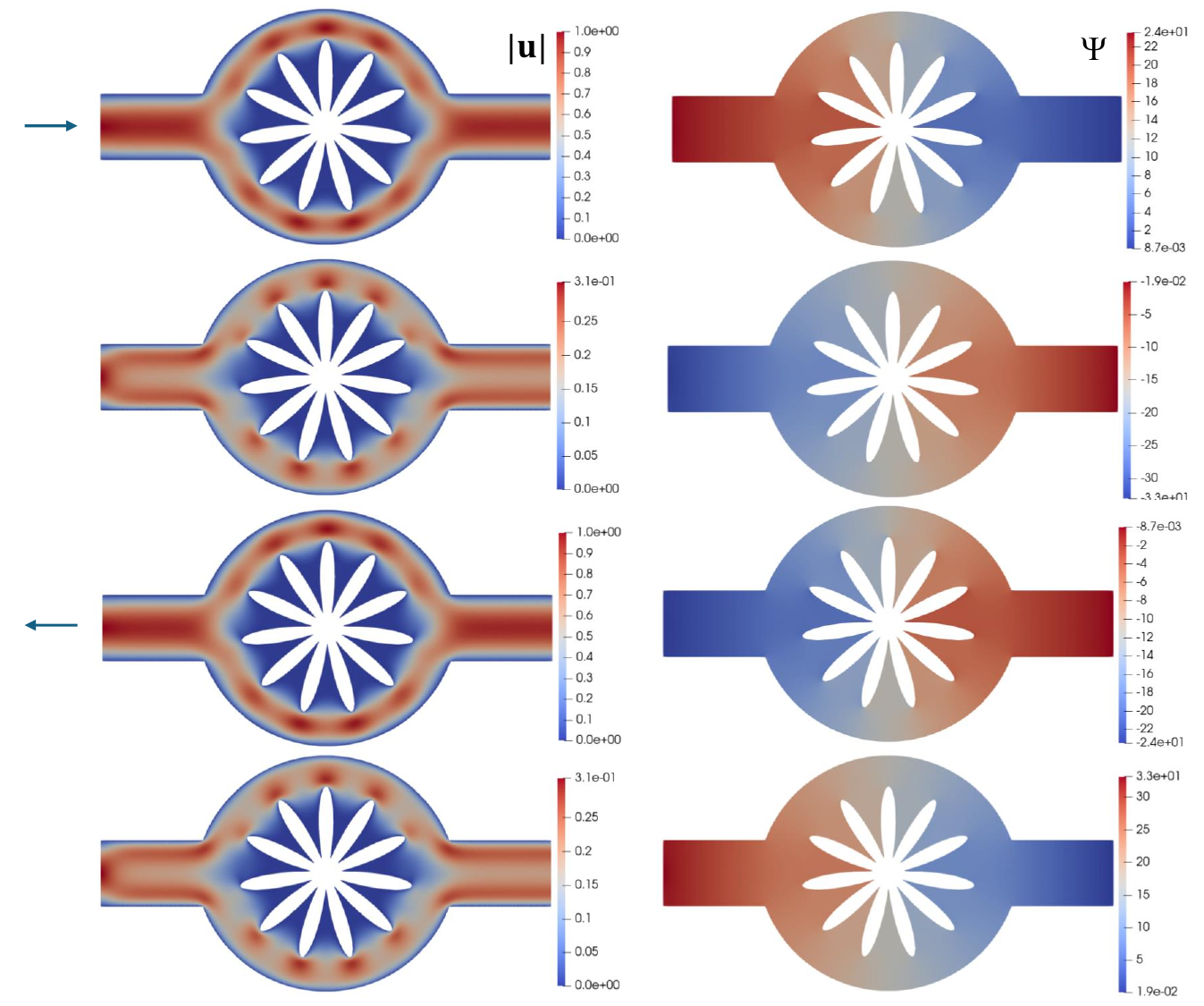}
\caption{Test 4. \(\alpha = 1.09\), dimensionless velocity and pressure fields every quarter-period. The arrow shows the mean flow direction.}
\label{fig_test4_plot_wom_1}
\end{figure}

\begin{figure}
\centering
\includegraphics[width=.9\textwidth]{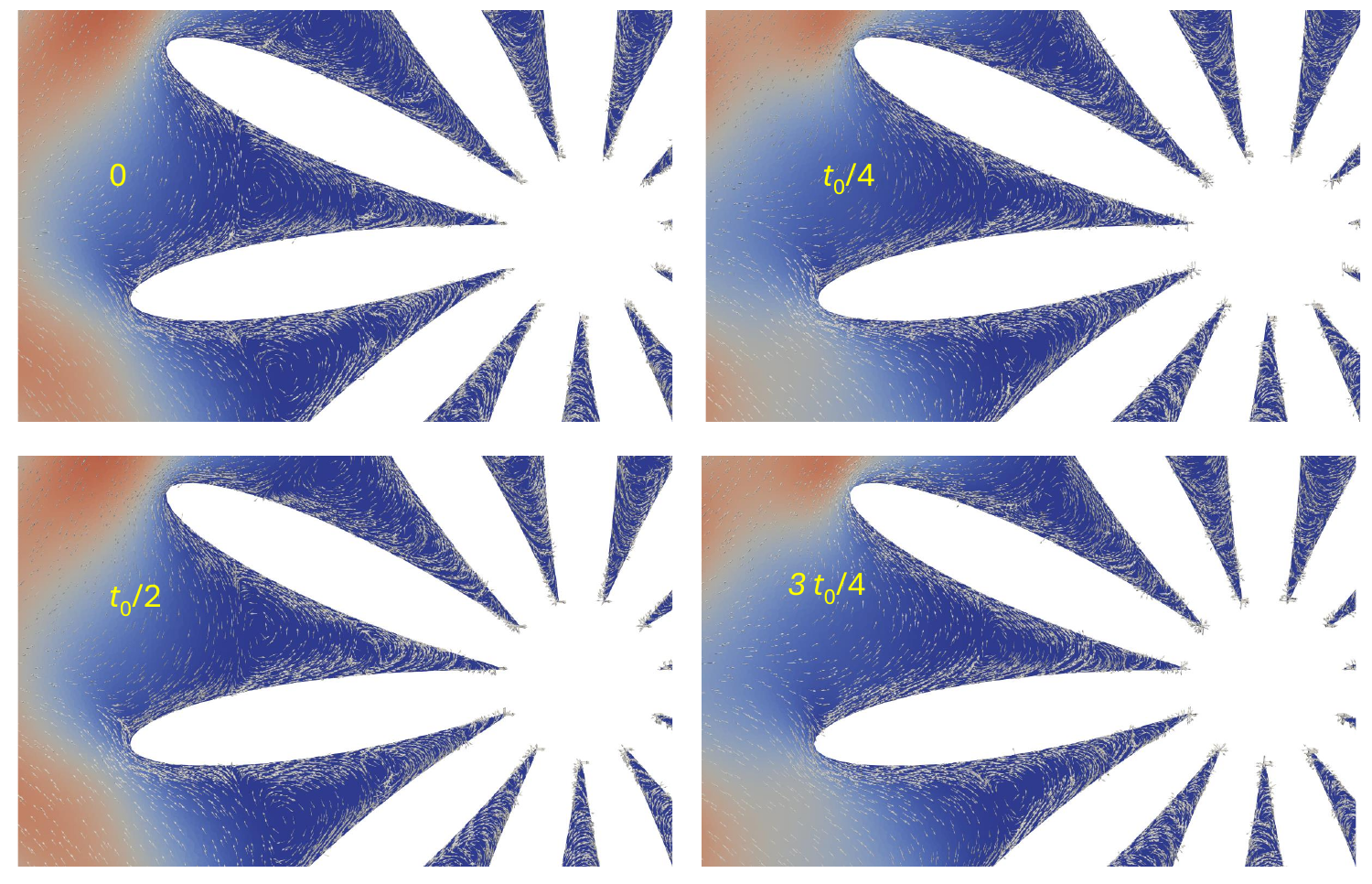}
\caption{Test 4. \(\alpha = 1.09 \), zoom of the velocity vectors close to the invested obstacle at every quarter-period.}
\label{fig_test4_plot_wom_1-2}
\end{figure}

\begin{figure}
\centering
\includegraphics[width=0.85\textwidth]{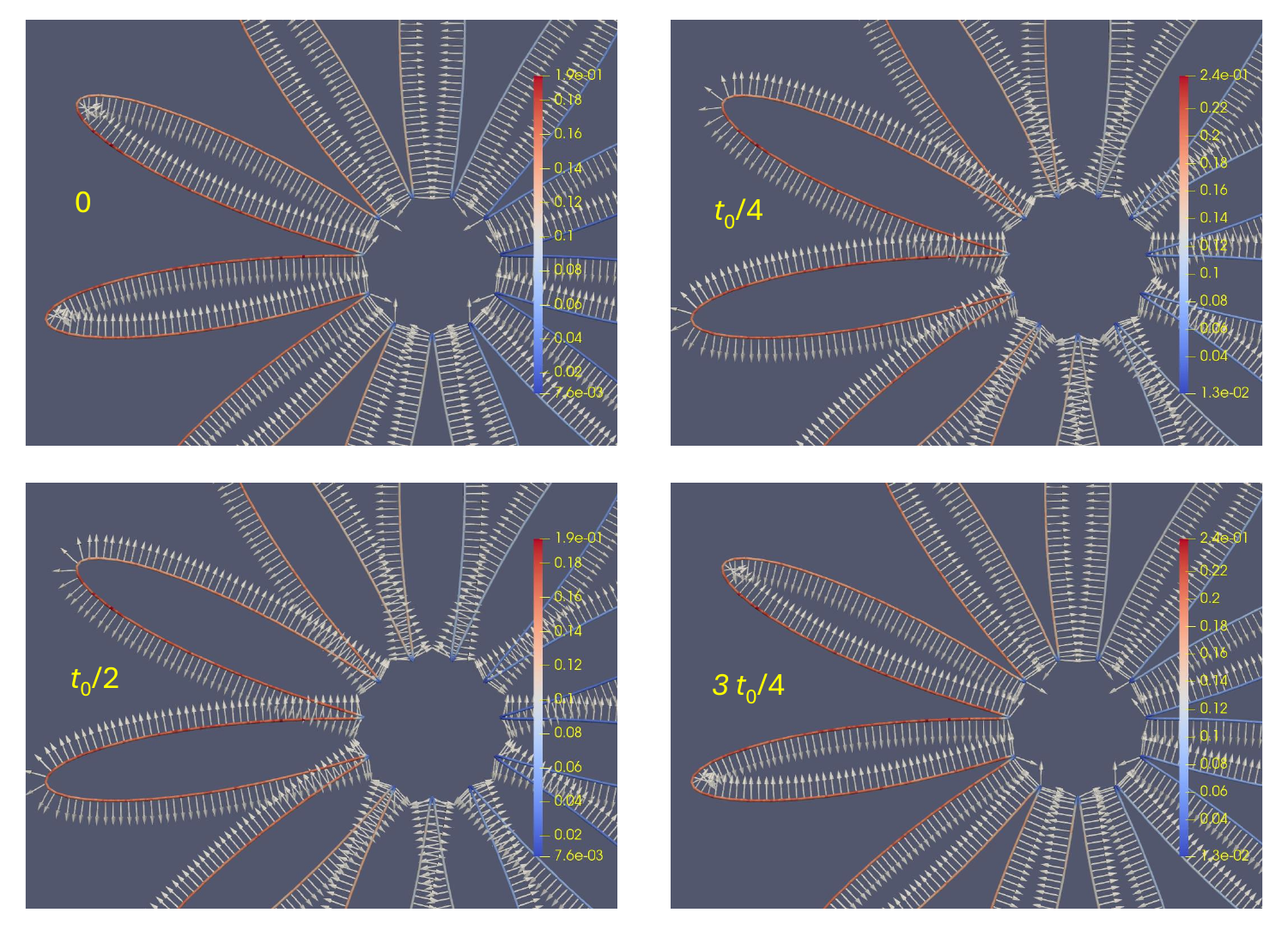}
\caption{Test 4. \(\alpha = 1.09 \), total forces acting on the perimeter of the obstacle at every quarter-period.}
\label{fig_test4_plot_wom_1-3}
\end{figure}

\begin{figure}
\centering
\includegraphics[width=0.85\textwidth]{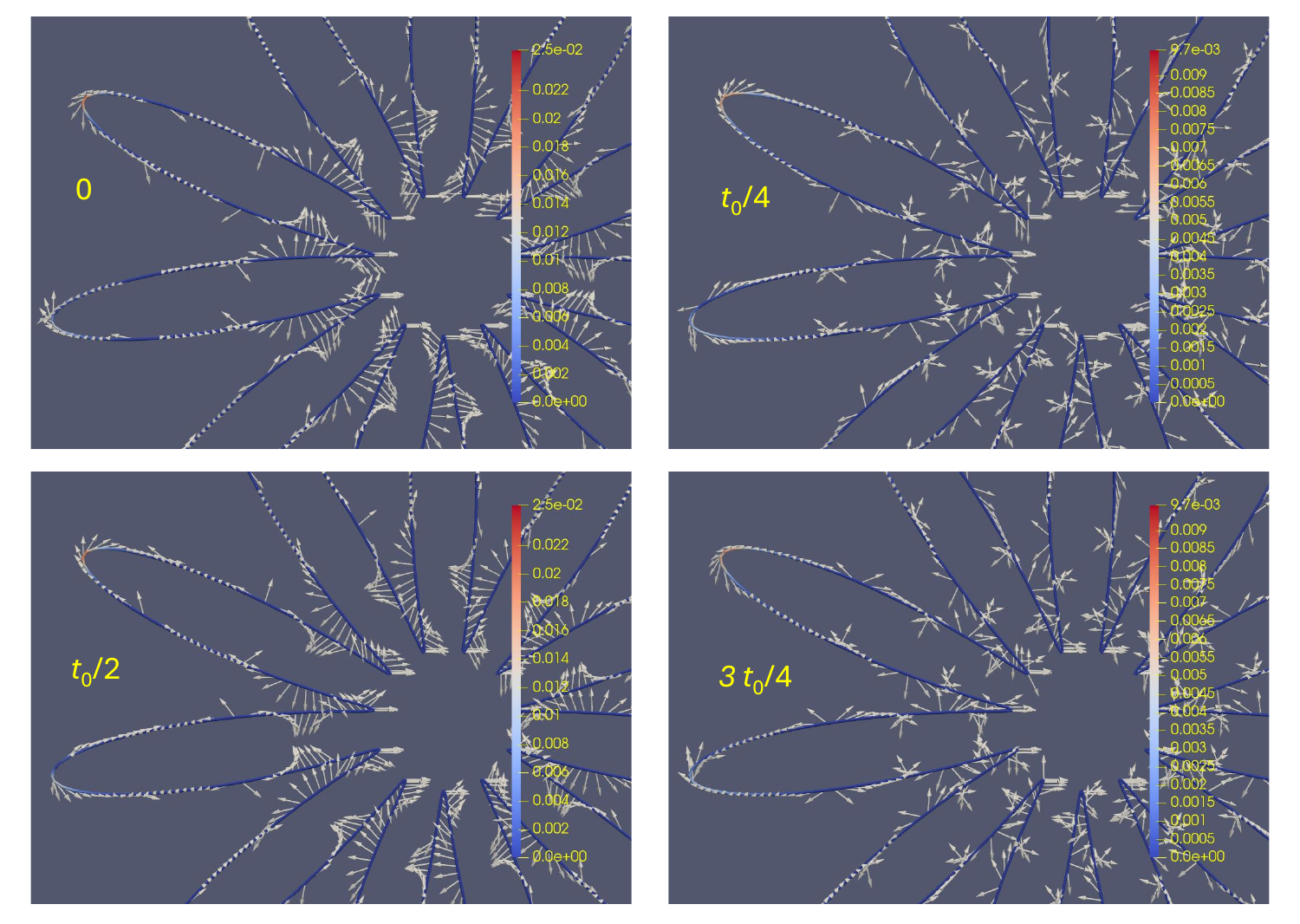}
\caption{Test 4. \(\alpha = 1.09 \), viscous forces acting on the perimeter of the obstacle at every quarter of period.}
\label{fig_test4_plot_wom_1-4}
\end{figure}

\section{Conclusions}\label{sec:concl}

We have developed $H$(div)-conforming mixed finite element methods for the unsteady incompressible Stokes equations. We applied a projection method in the framework of the incremental pressure correction methodology, where a predictor and a projection problems are sequentially solved, accounting for the viscous effects and incompressibility, respectively. The predictor problem is based on a stress-velocity mixed formulation, while the projection problem uses a velocity-pressure mixed formulation. We established unconditional stability and first order in time accuracy. We then developed a specific method of the family on generally unstructured triangular grids, using the multipoint flux mixed finite element methodology.
We used the \(RT_1\) mixed finite element pair and applied an inexact numerical integration to obtain mass lumping and local stress or flux elimination in both the prediction and projection problems. The resulting algebraic systems, which are sparse symmetric and positive definite with only three unknowns per element, are solved efficiently by the preconditioned conjugate gradient method. The scheme results in pointwise divergence-free velocity computed at the end of each time step. The computed velocity and pressure are second order accurate in space.
The presented numerical experiments illustrate the accuracy and efficiency of the method for several benchmark problems and a challenging problem with highly complex geometry. 

The proposed methodology could be straightforwardly extended to 3D problems. In this sense, the good performance obtained over the badly distorted grids, see \cref{test1}, is encouraging for further applications to 3D tetrahedral grids, where an optimal aspect ratio is difficult to achieve. Furthermore, as mentioned in the introduction, the proposed methodology could be regarded as the seed of further extensions, e.g., coupled flow and transport problems and fluid-structure interaction. In particular, due to the good performance and robustness shown by the multipoint flux mixed finite element method proposed in \cite{W-Y} for Darcy flow with heterogeneous and discontinuous full tensor permeability, a natural extension of the presented algorithm would be to include the Brinkman term in the momentum equations to simulate the interaction between a free fluid and a porous medium, see \cite{Arico-ODA}.    

\section*{Acknowledgments}
This work was partially supported by the German Research Foundation (DFG), by funding Sonderforschungsbereich (SFB) 1313 (Project Number 327154368, Research
Project A02), University of Stuttgart, and by funding SimTech via Germany’s Excellence Strategy (EXC
2075–390740016), University of Stuttgart. The third author was partially supported by SimTech Visiting Professorship, University of Stuttgart, and NSF grants DMS-2111129 and DMS-2410686.

\bibliography{stokes-mfe-proj}

\end{document}